\documentclass[sn-mathphys]{sn-jnl}

\usepackage{amsmath,amssymb,latexsym,pict2e,epic,graphicx, rotating}

\jyear{2021}%

\theoremstyle{thmstyleone}%
\newtheorem{theorem}{Theorem}
\theoremstyle{thmstyletwo}%
\newtheorem{remark}{Remark}%
\newtheorem{lemma}[theorem]{Lemma}%

\theoremstyle{thmstylethree}%

\def\PP{{\mathbb P}}
\def\CC{{\mathbb C}}
\def\RR{{\mathbb R}}
\def\ZZ{{\mathbb Z}}
\def\cO{{\mathcal O}}
\def\bp{{\boldsymbol p}}
\newcommand{\bn}{\boldsymbol{n}}
\newcommand{\eps}{{\varepsilon}}
\DeclareMathOperator{\Sing}{Sing}
\DeclareMathOperator{\Vers}{Vers}
\DeclareMathOperator{\bra}{br}
\DeclareMathOperator{\conj}{conj}
\DeclareMathOperator{\Span}{Span}
\DeclareMathOperator{\ord}{ord}
\DeclareMathOperator{\mt}{mult}
\DeclareMathOperator{\Sym}{Sym}
\DeclareMathOperator{\Tor}{Tor}
\DeclareMathOperator{\TCon}{TCon}
\DeclareMathOperator{\conv}{conv}

\DeclareMathOperator{\chara}{char}
\DeclareMathOperator{\Fract}{Fr}

\def\Sig{{\Sigma}}

\begin{document}

\title[Enumeration of non-nodal real plane rational curves]{Enumeration of non-nodal real plane rational curves}



\author*[1]{\fnm{Eugenii} \sur{Shustin}}\email{shustin@tauex.tau.ac.il}


\affil[1]{\orgdiv{School of Mathematical Sciences}, \orgname{Tel Aviv University}, \orgaddress{\street{Ramat Aviv}, \city{Tel Aviv}, \postcode{6997801}, \country{Israel}}}




\abstract{Welschinger invariants 
enumerate real nodal rational curves in the plane or in another real rational surface. We analyze the existence of 
similar enumerative invariants that count real rational plane curves having
prescribed non-nodal singularities and passing through a generic conjugation-invariant configuration of appropriately many points in the plane. We show that an invariant like this is unique: it enumerates real rational three-cuspidal quartics that pass through generically chosen four pairs of complex conjugate points. As a consequence, we show that through any generic configuration of four pairs of complex conjugate points, one can always trace a pair of real rational three-cuspidal quartics.}

\keywords{Real enumerative geometry, Plane curve singularities, Versal deformation, Equisingular family}


\pacs[MSC Classification]{Primary 14N10; Secondary 14H15; 14H20; 14P05}

\maketitle

\section{Statement of the problem and main results}\label{secbs1}

Enumeration of algebraic curves in algebraic varieties is a classical problem. A typical and very important example is the enumeration of plane curves of given degree and genus. The Severi variety $V_{d,g}$ that parameterizes irreducible complex plane curves of degree $d$ and genus $g\le\frac{1}{2}(d-1)(d-2)$ is irreducible of dimension $3d+g-1$ and its generic member is a nodal curve (see \cite{Sev} and \cite{Har}). Thus, whenever we take a configuration $\bp$ of $3d+g-1$ points in $\PP^2$ in general position, the set of the curves $C\in V_{d,g}$ passing through $\bp$ is finite, and the number of such curves does not depend on the choice of $\bp$.

In contrast, the number of real curves $C\in V_{d,g}$ passing through a conjugation-invariant configuration $\bp$ does depend on the choice of $\bp$.
For example, the number of real rational cubics passing through $8$ real points in general position varies from $8$ to $12$ (see, for instance, \cite[Page 788]{DK}). In his seminal paper \cite{Wel}, J.-Y. Welschinger proved, in particular, that the enumeration of real plane rational curves of a given degree $d$, passing through a general conjugation-invariant configuration $\bp$ of $3d-1$ points, does not depend on the choice of $\bp$, provided, we fix the number of real points and the number of complex conjugate pairs in $\bp$, and we count each curve with weight $\pm1$ (so-called {\it Welschinger sign}). More precisely, 
the weight of a real nodal curve $C$ is the product of contributions of all the nodes: a non-real node contributes factor $1$, a real hyperbolic node (the transverse intersection of two real smooth branches) contributes factor $1$, while a real elliptic node (the transverse intersection of two smooth complex conjugate branches) contributes factor $-1$.

It is natural to ask whether Welschinger's idea provides invariants in other real enumerative problems. Welschinger showed that the invariance extends to enumeration of real pseudoholomorphic curves in real rational symplectic four-folds \cite{Wel} and to enumeration of real rational curves in $\PP^3$ \cite{Wel2}. However,
Welschinger's signed count fails to be invariant of the choice of constraints in the following cases:
\begin{itemize}\item enumeration of real plane curves of a given degree and a given {\it positive} genus \cite[Section 3]{IKS0};
\item enumeration of real plane rational curves of a given degree, having one ordinary cusp and ordinary nodes as the rest of singularities \cite{Wel1}\footnote{In fact, Welschinger shows in \cite{Wel1} that one can keep the invariance by adding several correcting terms that count curves of rather different types.}.
\end{itemize}

The present paper aims to analyze in depth the latter situation and to answer the following

\smallskip{\bf Question:} {\it Does there exist a real enumerative invariant counting real plane rational curves of a given degree that have non-nodal singularities of prescribed types and nodes (if any) as the rest of singularities?}

\smallskip

Our {\bf answer} is: {\it In a reasonable setting, there are no such real enumerative invariants} except for a single case of three-cuspidal quartics.

\smallskip
To formulate the main results, we use some elements of the singularity theory of plane algebraic curves, see, for instance, \cite[Chapter I, Sections 2 and 3]{GLS0} or Section \ref{sec-sing} below, where we provide all necessary definitions, notations, and basic facts. 

For a sequence $S_1,...,S_r$ of (complex) topological types of isolated plane curve singular points, denote by $V_d(S_1,...,S_r)\subset\vert\cO_{\PP^2}(d)\vert$ the locus parameterizing complex 
plane curves of degree $d$ having $r$ singular points
of types $S_1,...,S_r$, respectively, and no other singularities. In general, it is the union of quasiprojective varieties and traditionally is called an {\it equisingular family}. We also introduce the sublocus $V_d^{irr}(S_1,...,S_r)\subset V_d(S_1,...,S_r)$ parameterizing irreducible curves. The elements of $V_d^{irr}(S_1,...,S_r)$ are rational whenever $\delta(S_1)+...+\delta(S_r)=\frac{1}{2}(d-1)(d-2)$. By $V_d^{\RR}(S_1,...,S_r)\subset V_d(S_1,...,S_r)$, resp. $V_d^{irr,\RR}(S_1,...,S_r)\subset V_d^{irr}(S_1,...,S_r)$ we denote the subsets parameterizing real curves. If $V_d(S_1,...,S_r)$ is smooth of pure complex dimension $n$, then $V^{\RR}_d(S_1,...,S_r)$ is smooth of pure real dimension $n$. For a nonnegative partition $n=n_{re}+2n_{im}$, denote by ${\mathcal P}_{n_{re},n_{im}}\subset\Sym^n(\PP^2)$ the space of configurations consisting of $n_{re}$ real points and $n_{im}$ pairs of complex conjugate points. Under the above dimnsion condition, a generically chosen configuration $\bp\in{\mathcal P}_{n_{re},n_{im}}$ defines a finite set
$$V^{\RR}_d(S_1,...,S_r,\bp)=\left\{C\in V^{\RR}_d(S_1,...,S_r)\ :\ C\supset\bp\right\}.$$

We want to enumerate the elements of $V_d^{irr,\RR}(S_1,...S_r,\bp)$ in an invariant way with respect to the choice of $\bp$. For that, we introduce Welschinger-like weights defined as follows. By a real plane curve singularity we call either the germ of a real plane curve at a real singular point, or a pair of germs of a real plane curve at two complex conjugate singular points. By a real topological type of a real plane curve singularity we mean the equivalence class of such real singularities up to an equivariant homeomorphism of a neighborhood\footnote{For a pair of non-real complex conjugate singular points, the real topological type is just the complex topological type of one of these points.}. Denote by $\RR TT$ the set of all real topological types of real plane curve singularities. A {\it real singularity weight} is a map
$$\varphi:\RR TT\to R\setminus\{0\},$$
where $R$ is an integral domain whose fraction field satisfies $\chara\Fract(R)\ne2$. For a real singularity weight $\varphi$, we define the {\it Welschinger-like weight} of a real irreducible plane curve $C$:
$$w_\varphi(C)=\prod_{z\in\Sing(C)\cap\PP^2_\RR}\varphi(C,z)\cdot\prod_{z,\conj z\in\Sing(C)\setminus\PP^2_\RR}\varphi(C,\{z,\conj z\}).$$
For example, we obtain the original Welschinger sign by choosing $R=\ZZ$ and setting
$$\varphi(A_1^e)=-1,\quad\varphi(A_1^h)=1,\quad \varphi(2A_1^{im})=1,$$
where $A_1^e$ and $A_1^h$ denote the types of an elliptic and a hyperbolic node, respectively, and $(aA_1^{im})$ is a pair of non-real nodes.

\begin{theorem}\label{t-rev1}
With a single exception specified in Theorem \ref{t-rev2}, for any integer $d\ge4$, any smooth family of non-nodal rational curves $V^{irr}_d(S_1,...,S_r)$ of pure dimension $n$, any nonnegative partition $n=n_{re}+2n_{im}$, and any real singularity weight $\varphi$, that satisfy some extra mild restrictions, the value
$$\sum_{C\in V^{irr,\RR}_d(S_1,...,S_r,\bp)}w_\varphi(C)$$
{\bf does depend} on the choice of a generic configuration $\bp\in{\mathcal P}_{n_{re},n_{im}}$.
\end{theorem}

Precise statements are presented in Theorems \ref{tbs3new}, \ref{tbs1new}, and \ref{tbs2new} in Sections \ref{secbs8new}, \ref{secbs11}, and \ref{secbs12}, respectively.

\begin{theorem}\label{t-rev2}
The family $V_4^{irr}(3A_2)$ is smooth of the expected dimension $8$. The number $$W(4,3A_2;(0,4)):=\#V^{irr,\RR}_4(3A_2,\bp),\quad\text{where}\quad\bp\in{\mathcal P}_{0,4},$$ is invariant for the choice of a generic configuration $\bp\in{\mathcal P}_{0,4}$. Furthermore, $W(4,3A_2;(0,4))$ is a positive even integer. In particular, through any generic configuration of four pairs of complex conjugate points, one can trace at least two real three-cuspidal quartics.
\end{theorem}

In this example, the real singularity weight is just $\varphi\equiv1\in\ZZ$. We do not know how to compute the value of the invariant $W(4,3A_2;(0,4))$. However, it is easy to see that it is positive: this follows from the invariance stated in the above theorem and from the fact that one can take a real three-cuspidal quartic, choose a configuration of four pairs of complex conjugate points on it, and then deform the quartic keeping three cusps and moving the configuration into a general position. At last, we note that $W(4,3A_2;(0,4))$ has the same parity as the number of the complex three-cuspidal quartics passing through $8$ points in general position, and the latter number equals $140$ (see \cite[Example 10.2]{Ka}).

\begin{remark}\label{r-rev1} It is worth noticing that, in contrast to the global situation addressed in Theorem \ref{t-rev1}, in the local setting, there exist an infinite series of enumerative invariants counting real nodal-cuspidal deformations of real curve singularities \cite[Propositions 4.1 and 4.3]{Sh20}.
\end{remark}

The strategy of the proof of Theorems \ref{t-rev1} and \ref{t-rev2} is similar to the argument used in \cite{IKS,Wel,Wel1} and consists in the analysis of wall-crossing events along a path joining two point constraints. These wall-crossing events are described in Section \ref{secbs3}. As example, in Section \ref{secbs2}, we consider real cuspidal cubics and describe two wall-crossing events breaking the invariance of enumeration (cf. \cite{Wel1}). In Sections \ref{secbs8new}, \ref{secbs11}, and \ref{secbs12}, we prove three statements on the failure of enumeration invariance that together cover the range of Theorem \ref{t-rev1}. In Section \ref{secbs4}, we prove Theorem \ref{t-rev2}.
At last, in Section \ref{secbs6new} we
demonstrate an example of a slightly modified real enumerative problem, in which the count of real curves appears to be invariant.
A generalization of this observation will be considered in a forthcoming paper.

\section{Preliminary information on plane curves and their singularities}\label{sec-sing}
To make the reading of the main body of the paper easier, in this section we recall necessary definitions and facts on plane curves, their singularities, and versal deformations (unfoldings) of singular points. A systematic theory of this stuff can be found in various sources. The closest ones to our setting are: \cite[Sections I.2 and I.3]{GLS0} (plane curves singularities and their invariants) and \cite[Sections 1, 3, 4, and 5]{DH}, \cite[Chapter II]{GLS0} (versal deformations and their stratification).

{\bf(1)} Let $(C,z)\subset\PP^2$ be a germ of a plane curve $C$ at an isolated singular point $z$. In the ball $B(C,z)\subset\PP^2$ of a sufficiently small radius centered at $z$ (called {\it Milnor ball}),
the intersection $(C,z)\cap B(C,z)$ is a bouquet of topological disks with the common center $z$, smooth outside $z$, and transversally intersecting $\partial B(C,z)$ along their boundaries. Each of these disks is called a {\it local branch} of $C$ centered at $z$. If $f(x,y)=0$ is a local equation of $(C,z)$, $z=(0,0)$, then the local branches are given by $f_i(x,y)=0$, where $f=f_1\cdot...\cdot f_r$ is the decomposition of $f$ into irreducible factors in the ring $\CC\{x,y\}$ of germs of holomorphic functions.

Two germs $(C,z), (C',z')\subset\PP^2$ are called {\it topologically equivalent} if there is a homeomorphism $B(C,z)\overset{\sim}{\to}B(C',z')$ taking $z$ to $z'$ and $(C,z)\cap B(C,z)$ to $(C',z')\cap B(C',z')$. A topological equivalence class is usually called {\it topological type}.

\smallskip{\bf(2)} In the sequel, we use several classical topological invariants of a plane curve germ $(C,z)$ given by a local affine equation $f(x,y)=0$.
\begin{itemize} \item The number of local branches $\bra(C,z)$.
\item The multiplicity $\mt(C,z)$ which is the minimal degree of a monomial in the Taylor expansion of $f(x,y)$ at $z=(0,0)$.
\item The Milnor number $\mu(C,z)=\dim_\CC\CC\{x,y\}/\langle \partial f/\partial x,\partial f/\partial y\rangle$.
\item The $\delta$-invariant $\delta(C,z)$ which can be defined as the maximal number of nodes that appear in a small deformation of the germ $(C,z)$. It is well-known that the geometric genus equals
    $$g(C)=\frac{(d-1)(d-2)}{2}-\sum_{i=1}^r\delta(S_i)\quad\text{for any curve}\quad C\in V^{irr}_d(S_1,...,S_r),$$
    and that $\mu(C,z)$ and $\delta(C,z)$ are related as follows:
    $$\mu(C,z)=2\delta(C,z)-\bra(C,z)+1.$$
    \item The $\varkappa$-invariant $\varkappa(C,z)$ which is the intersection number at $z$ of $C$ with its generic polar curve $\{\alpha\partial f/\partial x+\beta\partial f/\partial y=0\}$, $[\alpha:\beta]\in\PP^1$. It appears in the generalized Pl\"ucker formula for the degree of the dual curve
        $$d^*=d(d-1)-\sum_{i=1}^r\varkappa(S_i) \quad\text{for any curve}\quad C\in V^{irr}_d(S_1,...,S_r).$$
       The $\varkappa$-invariant can be computed via the other invariants:
        $$\varkappa(C,z)=\mu(C,z)+\mt(C,z)-1=2\delta(C,z)+\mt(C,z)-\bra(C,z).$$
\end{itemize}

\smallskip{\bf(3)} Every isolated plane curve singular point $(C,z)$ admits a versal deformation, i.e., a deformation which induces any other deformation. Each versal deformation is a cylinder over the semi-universal deformation
$$\begin{matrix}(C,z) & \hookrightarrow & \widetilde\Vers(C,z)\\ 
\downarrow & & \downarrow\pi\\ 0 & \hookrightarrow & \Vers(C,z)\end{matrix}$$
with the base $\Vers(C,z)$ isomorphic to the germ at $0$ of the space ${\mathcal O}_{C,z}/J(f)$,
where $J(f)=\langle\partial f/\partial x,\partial f/\partial y\rangle$ is the Jacobian ideal. The fibers $\pi^{-1}(b)$, $b\in\Vers(C,z)$, can be regarded as plane algebraic curves restricted to the Milnor ball $B(C,z)$ which all are smooth along $\partial B(C,z)$ and intersect $\partial B(C,z)$ transvesally. Note also that for simple ADE singularities, $\dim\Vers(C,z)=\mu(C,z)$. The semi-universal deformation is often represented as a {\it (uni)versal unfolding}
$$F(x,y,t_1,...,t_k)=f(x,y)+\sum_{i=1}^kt_ih_i,\quad (t_1,...,t_k)\in(\CC^k,0)\simeq\Vers(C,z),$$
where $h_1,...,h_k\in{\mathcal O}_{C,z}$ induce a basis of ${\mathcal O}_{C,z}/J(f)$, and
$$B(C,z)\times\Vers(C,z)\supset\{F(x,y,t_1,...,t_k)=0\}\simeq\widetilde\Vers(C,z)$$
is a fiberwise isomorphism.

The equisingular locus $ES(C,z)\subset\Vers(C,z)$ parameterizes elements $\widetilde\Vers(C,z)_b=\pi^{-1}(b)$, $b\in\Vers(C,z)$, that are topologically equivalent to $(C,z)$. It is the germ of a smooth analytic subvariety with the tangent space $I^{es}(C,p)/J(f)$, where $I^{es}(C,z)\subset{\mathcal O}_{C,z}$ is the so-called equisingular ideal. Note that for simple ADE singular point, $I^{es}(C,z)=J(f)$.

The number $$c(C,z)=\dim_\CC{\mathcal O}_{C,z}/I^{es}(C,z)$$ (called an {\it expected codimension}) is a topological invariant. Roughly speaking, it counts the number of conditions imposed on plane curves by a singularity of a given topological type. For example, for simple ADE singularities, $ES(C,z)=0\in\Vers(C,z)$, and $c(C,z)=\mu(C,z)$.

Some other loci in $\Vers(C,z)$ are used in the sequel. Denote by $\Sigma(S_1,...,S_k)\subset\Vers(C,z)$ the locus parameterizing elements $\widetilde\Vers(C,z)_b$, $b\in\Vers(C,z)$, which have exactly $k$ singular points of topological types $S_1,...,S_k$, respectively. This is a locally closed analytic subvariety of expected codimension $\sum_{i=1}^kc(S_i)$.

\smallskip{\bf(4)} The {\it expected dimension} of an equisingular family $V_d(S_1,...,S_r)$ is defined to be
 $$\dim\vert\cO_{\PP^2}(d)\vert-\sum_{i=1}^rc(S_i)=\frac{d(d+3)}{2}-\sum_{i=1}^rc(S_i).$$
The actual dimension of $V_d(S_1,...,S_r)$ equals the expected one, if the family $V_d(S_1,...,S_r)$ is {\it $T$-smooth}. We recall the latter notion here. Let $C\in V_d(S_1,...,S_r)$ have singular points $z_1,...,z_r$ of topological types $S_1,...,S_r$, respectively. Denote by $\widehat V_d(S_i,z_i)$, $i=1,...,r$, the germ at $C$ of the family of plane curves of degree $d$ having a singular point of type $S_i$ in a small neighborhood of $z_i$ and maybe some other singular points. We say that the germ $\widehat V_d(S_i,z_i)$ is $T$-smooth if it is smooth of codimension $c(S_i)$ in $\vert\cO_{\PP^2}(d)\vert$ and we say that the family $V_d(S_1,...,S_r)$ is $T$-smooth at $C$ if all the germs $\widehat V_d(S_i,z_i)$, $i=1,...,r$, are $T$-smooth and they intersect transversally at $C$. Correspondingly, we say that the whole family $V_d(S_1,...,S_r)$ is $T$-smooth if it is $T$-smooth at each $C\in V_d(S_1,...,S_r)$. In terms of the preceding item (3), the tangent space at $C$ to a $T$-smooth locus $V_d(S_1,...,S_r)$ can be identified with
$$H^0(C,{\mathcal J}_{Z^{es}(C)/C}\otimes{\mathcal O}_C(d)),$$
where the zero-dimensional subscheme $Z^{es}(C)\subset C$ is supported at $\Sing(C)$ and for each $z\in Sing(C)$, is defined by the ideal $I^{es}(C,z)\subset{\mathcal O}_{C,z}$, while ${\mathcal J}$ stands for the ideal sheaf.

There are several numerical conditions ensuring the $T$-smoothness of families of plane curves (see \cite[Sections 4.3, 4.4, and 4.6]{GLS} for a complete account). We recall here a few particular instances of such conditions restricted to curves with simple ADE singularities.

\begin{lemma}\label{TS}
Let $S_1,...,S_r$ be simple ADE singularity types.

(1) If
\begin{equation}\sum_{i=1}^r(\mu(S_i)-1)<3d,\label{eTS1}\end{equation}
then the family $V_d^{irr}(S_1,...,S_r)$ is either empty, or $T$-smooth.

(2) Let $C\in V_d^{irr}(S_1,...,S_r)$ have singular points $z_1,...,z_r$ of types $S_1,...,S_r$, respectively. Consider the germ at $C$ of the locus $$V^{irr}_d(S_1^{fix},...,S_k^{fix}, S_{k+1},...,S_r):=
\{C'\in V^{irr}_d(S_1,...,S_r)\ :\ z_1,...,z_k\in\Sing(C')\}$$
(that is, $(C',z_i)$ is of type $S_i$ for each $i=1,...,k$). If
\begin{equation}\sum_{i=1}^k\lambda(S_i)+\sum_{i=k+1}^r(\mu(S_i)-1)<3d,\label{eTS2}\end{equation} where $\lambda(S)$ stands for the invariant listed at \cite[The last row in the table on page 443]{GLS},
then the locus $V_d^{irr}(S_1^{fix},...,S_k^{fix},S_{k+1},...,S_r)$ is smooth of dimension
$$\frac{d(d+3)}{2}-\sum_{i=1}^k(\mu(S_i)+2)-\sum_{i=k+1}^r\mu(S_i).$$
\end{lemma}

The next lemma concerns certain families of reducible curves too.

\begin{lemma}\label{TS1}
Let $S_1,...,S_r$ be simple ADE singularity types.

(1) If
\begin{equation}\sum_{i=1}^r\mu(S_i)<4d-4,\label{eTS3}\end{equation}
then the family $V_d(S_1,...,S_r)$ is either empty, or $T$-smooth.

(2) Let $C\in V_d(S_1,...,S_r)$ and $C$ split into irreducible components $C_1,C_2$ of degrees $d_1,d_2$, respectively, so that $C_1\cap C_2\cap (\Sing(C_1)\cup\Sing(C_2))=\emptyset$. If
\begin{equation}\sum_{z\in\Sing(C_j)}(\mu(C_j,z)-1)+d_1d_2-\#(C_1\cap C_2)<3d_j,\quad j=1,2,\label{eTS4}\end{equation}
then the locus $V_d(S_1,...,S_r)$ is $T$-smooth at $C$.
\end{lemma}

At last, we give sufficient conditions for the versality of deformations induced by families of plane curves.

\begin{lemma}\label{TS2}
Let $S_1$ be a simple ADE singularity type. Suppose that $V_d(S_1,...,S_r)$ is $T$-smooth at $C\in V_d(S_1,...,S_r)$ and that $z_1,...,z_r$ are the singular points of $C$ of types $S_1,...,S_r$, respectively. Then the germ at $C$ of the family
$$\bigcap_{i=2}^r\widehat V_d(S_i,z_i)$$
induces a versal deformation of the germ $(C,z_1)$.
\end{lemma}

\smallskip
{\bf(5)}
In case of real plane curves, we consider {\it real topological types} of singularities:
\begin{itemize}\item two germs $(C,z)$, $(C',z')$ of real plane curves $C$ and $C'$ at real singular points $z$ and $z'$ are of the same real topological type, if there exists an equivariant homeomorphism of their Milnor balls $B(C,z)\overset{\sim}{\to}B(C',z')$ taking $C\cap B(C,z)$ to $C'\cap B(C',z')$;
\item the real topological type of a pair of non-real complex conjugate singular points $z,\conj z$ of a real plane curve $C$ is the (complex) topological type of the germ $(C,z)$ (which, of course, is the same as for $(C,\conj z)$).
\end{itemize}
For example, real unibranch singularities of the complex topological type $A_{2k}$ ($k\ge1$) are of the same real topological type, while real singularities of the complex topological type $A_{2k-1}$ ($k\ge1$) split into two real topological types $A^e_{2k-1}$ and $A^h_{2k-1}$ according as the two local branches are complex conjugate or both are real.

We say that two curves $C_1,C_2\in V_d^\RR(S_1,...,S_r)$ are {\it $\RR$-isosingular}, if they have the same collection of real topological types of their singular points.

\smallskip{\bf(6)} Given the germ $(V,p)\subset\CC^N$ of an analytic variety, we speak of the geometric tangent cone $\TCon_pV$ as the closure in $T_p\CC^N$ of the union of the limits of tangent spaces $T_{p'}V$ at a sequence of smooth points $p'\in V$ converging to $p$. For example, if $p$ is a smooth point then $\TCon_pV=T_pV$, if $V$ is one-dimensional then $\TCon_pV$ is the union of lines. 

\smallskip{\bf(7)} We recall the following rather simple statement well-known to the experts, which we shall use in the construction of certain singular plane curves. For the reader's convenience we sketch a proof.

\begin{lemma}\label{lbs7}
Let $\Tor(P)$ be a toric surface associated with a non-degenerate convex lattice polygon $P\subset\RR^2$, and ${\mathcal L}_P$ the tautological line bundle on $\Tor(P)$. Then there exists a real rational nodal curve $C\in\vert{\mathcal L}_P\vert$, which is smooth along the toric divisors and has prescribed distribution of intersection multiplicities with toric divisors.
\end{lemma}

\begin{proof}
Let $P^1$ be the set of sides of $P$, $n_\sigma=c_1({\mathcal L}_P)[\Tor(\sigma)]$ the lattice length of the side $\sigma\in P^1$, $n_\sigma=\sum_in_{\sigma,i}$ a partition, $\vec{a}_\sigma=(a_{\sigma,x},a_{\sigma,y})\in\ZZ^2$ the primitive integral inner normal to $\sigma$. Then the desired real nodal curve is given by a parametrization
$$x=\prod_{\sigma\in P^1}\prod_i(t-\gamma_{\sigma,i})^{n_{\sigma,i}a_{\sigma,x}},\quad y=\prod_{\sigma\in P^1}\prod_{i}(t-\gamma_{\sigma,i})^{n_{\sigma,i}a_{\sigma,y}},$$
where $\{\gamma_{\sigma,i}\}_{\sigma,i}$ is a generic collection of real numbers.

It is enough to exhibit just one nodal rational curve in $\Tor(P)$ with the given intersection profiles along toric divisors. One can construct it using the version of Mikhalkin' correspondence theorem \cite{Mik} suggested in \cite[Sections 3 and 5]{Sh06} (see also \cite[Section 2.5]{IMS}). Namely, one takes a family of rational curves in $\Tor(P)$ parameterized by a punctured disk and extends it to the central point so that the central surface splits into the union of toric surfaces governed by a plane trivalent rational tropical curve, while the central (algebraic) curve is the union of rational components embedded into components of the central surface. Then one verifies that each component of the central curve is nodal \cite[Lemma 3.5]{Sh06} and that each intersection point of two components of the central curve deforms into a nodal (or smooth) fragment of the nearby fiber of the original family \cite[Lemma 3.9]{Sh06}.
\end{proof}

\section{Three types of wall-crossing events}
\label{secbs3}

Let $V_d(S_1,...,S_r)$ be $T$-smooth of dimension $n=n_{re}+2n_{im}$. Consider a generic path $\{\bp(t)\}_{0\le t\le 1}\subset{\mathcal P}_{n_{re},n_{im}}$. Along the complement $[0,1]\setminus F$ of a finite subset $F$, the natural projection
$$\bigcup_{t\in[0,1]\setminus F}\big(V_d^{\RR}(S_1,...,S_r,\bp(t)),t\big)\to[0,1]\setminus F$$
is a smooth unramified covering. Here, we study changes (wall-crossing events) in the set  $V_d^{\RR}(S_1,...,S_r,\bp(t))$ when $t$ moves through a point of $F$. We do not consider all possible walls, but select only those we use in the proof of the main results.

\subsection{Wall-crossing: Degeneration of singularities}\label{secbs6}
{\bf(1)} Let $(C,z)$ be a germ of a real plane curve with a real isolated singular point $z$ of a simple topological type $S$, i.e., $S\in\{A_k,k\ge1,D_k,k\ge4,E_k,k=6,7,8\}$.
Let $\Sigma(S_1,...,S_k)\subset\Vers(C,z)$ be one-dimensional, and let the closure $\overline\Sig(S_1,...,S_k)\subset\Vers(C,z)$ be an irreducible analytic curve germ at $0$.
In particular, $\TCon_0\overline\Sigma(S_1,...,S_k)$ is a line. Pick a real (i.e., conjugation invariant) hyperplane $H\subset\Vers(C,z)$ transversally intersecting $\TCon_0\overline\Sigma(S_1,...,S_k)$ at $0$. Consider the family of intersections of $\Sigma_\RR(S_1,...,S_k)$ with real affine hyperplanes parallel to $H$. We call this family of intersections the {\it bifurcation of} $(C,z)$ of type $(S_1,...,S_k\Longrightarrow S)$.

\begin{lemma}\label{local-global}
In the notation and under the assumptions of the preceding paragraph, let $C\in V^\RR_d(S,S_{k+1},...,S_r)$ have singular points $z,z_{k+1},...,z_r$ of types $S,S_{k+1},...,S_r$, respectively, and let $V_d(S,S_{k+1},...,S_r)$ be $T$-smooth at $C$. Suppose also that
the germ $(V_d(S,S_{k+1},...,S_r),C)$ is contained in the closure $\overline V_d(S_1,...,S_k,S_{k+1},...,S_r)$, where $V_d(S_1,...,S_k,S_{k+1},...,S_r)$ is $T$-smooth of dimension $n=n_{re}+2n_{im}$.

If $\bp\in{\mathcal P}_{n_{re},n_{im}}$ and $\bp$ lies in general position on $C$, then 
there exists a path $\{\bp(t)\}_{-\eps<t<\eps}\subset{\mathcal P}_{n_{re},n_{im}}$ such that $\bp(0)=\bp$, and the germ at $C$ of the family $\{V_d^{\RR}(S_1,...,S_r,\bp(t))\}_{0<\vert t\vert<\eps}\to t\in(-\eps,\eps)\setminus\{0\}$ is fiberwise homeomorphic to the bifurcation of $(C,z)$ of type $(S_1,...,S_k\Longrightarrow S)$.
\end{lemma}

\begin{proof}
Taking into account that the germ of $\bigcap_{i=k+1}^r\widehat V_d(S_i,z_i)$ at $C$ induces a versal deformation of the curve germ $(C,z)$ (see Lemma \ref{TS2}), we only need to show that the linear system
\begin{equation}\Lambda_d(\bp):=\{C'\in\vert \cO_{\PP^2}(d)\vert\ :\ \bp\subset C'\}\label{Lambda}\end{equation}
has codimension $n$ and intersects transversally with $\TCon_C\overline V_d(S_1,...,S_r)$. Note that the required property does not depend on the real structure and can be considered within complex algebraic geometry. This we pick one-by-one $n-1$ points $p_1,...,p_{n-1}\in C$ so that $\Lambda_d(\{p_i\}_{i=1}^{n-1})$ has codimension $n-1$ and intersects transversally $T_CV_d(S,S_{k+1},...,S_r)$. This linear system intersects $\TCon_C\overline V_d(S_1,...,S_r)$ along a line spanned by $C$ and some other curve $C'$ of degree $d$. Then we define $\bp=\{p_i\}_{i=1}^n$ choosing the last point $p_n\in C\setminus C'$.
\end{proof}

\medskip\noindent{\bf(2)} {\it Deformations of $A_m$ singularities.} Recall that $A_m$ is the topological type of a complex plane curve singularity $y^2+x^{m+1}=0$, and that, for even $m$, there is unique real topological type also denoted by $A_m$ and for odd $m$, there are two real topological types $A^e_m$ represented by $y^2+x^{m+1}=0$ and $A^h_m$ represented by $y^2-x^{m+1}=0$.

\begin{lemma}\label{lbs1}
Let $(C,z)$ be a plane curve singularity of type $A_m$, $m\ge2$. Then

(1) The one-dimensional equisingular loci in $\Vers(C,z)$ are $\Sig(A_{m-1})$ and $\Sig(A_i,A_j)$, where $1\le i\le j$, $i+j=m-1$.

(2) The closures $\overline\Sig(A_{m-1})$ and $\overline\Sig(A_i,A_j)$, $1\le i<j$, $i+j=m-1$, are curve germs with a singularity of type $A_2$ at the origin and the tangent cone a line. If $m=2k+1$, $k\ge1$, then the closure $\overline\Sig(2A_k)$ is a smooth curve germ.

(3) Suppose that $(C,z)$ is real and $m$ is even. Then one half of $\Sig_\RR(A_{m-1})$ parameterizes curves with a real singularity of type $A^e_{m-1}$, while the other half parameterizes curves with a real singularity of type $A^h_{m-1}$. If $i,j\ge1$, $i+j=m-1$, $i$ is odd, $j$ is even, then one half of $\Sig_\RR(A_i,A_j)$ parameterizes curves with two real singular points of types $A^e_i$, $A_j$, while the other half parameterizes curves with two real singular points of types $A^h_i$, $A_j$.

(4) Suppose that $(C,z)$ is real and $m=2k+1$, $k\ge1$. Then
\begin{itemize}\item The locus $\Sig_\RR(A_{m-1})$ parameterizes curves with one real singular point of type $A_{m-1}$.
\item The locus $\Sig_\RR(A_i,A_j)$, where $i<j$ are odd, parameterizes curves with two real singular points of types $A^e_i$, $A^e_j$, resp. $A^h_i$, $A^h_j$, according as $(C,z)$ is of type $A^e_m$, resp. $A^h_m$.
    \item The locus $\Sig_\RR(A_i,A_j)$, where $i<j$ are even, parameterizes curves with two real singular points of types $A_i$ and $A_j$.
    \item If $k$ is even, then one half of the locus $\Sig_\RR(2A_k)$ parameterizes curves with two real points of type $A_k$, while the other half parameterizes curves with a pair of non-real complex conjugate singular points. If $k$ is odd, then one half of the locus $\Sig_\RR(2A_k)$ parameterizes curves with two real singular points of type $A^e_k$ (resp. $A^h_k$) of $(C,z)$ is of type $A^e_m$ (resp. $A^h_m$), while the other half parameterizes curves with two non-real complex conjugate singular points.
\end{itemize}
\end{lemma}

\begin{proof}
(1) This is a well-known fact, see \cite[Theorem 4]{Lya}.

(2) The universal unfolding of the singular point $z=(0,0)$ of type $A_m$ is $$y^2+x^{m+1}+\sum_{k=0}^{m-1}\lambda_kx^k=0,\quad(\lambda_0,...,\lambda_{m-1})\in(\CC^n,0).$$
Then the loci $\overline\Sig(A_{m-1})$ and $\overline\Sig(A_i,A_j)$, $1\le i<j$, $i+j=m-1$, parameterise curves
$$y^2+(x-mt)(x+t)^m=0\quad\text{and}\quad y^2+(x-(j+1)t)^{i+1}(x+(i+1)t)^{j+1}=0,\quad t\in(\CC,0),$$
respectively. In both cases we have
$$\lambda_k=a_kt^{n+1-k},\quad a_k\ne0,\quad k=0,...,n-1,$$
and the second statement follows.

(3) Let $m=2k$, $k\ge1$, and $(C,z)=\{y^2+x^{2k+1}=0\}$, $z=(0,0)$. Then (see item (2))
$$\overline\Sig_\RR(A_{m-1})=\{y^2+(x-mt)(x+t)^n=0,\ t\in(\RR,0)\},$$
and hence one obtains the singularity $A_{m-1}^h$ for $t>0$, and the singularity $A_{m-1}^e$ for $t<0$. The same argument settles the case of the locus $\Sig_\RR(A_i,A_j)$.

(4) Let $m=2k+1$, $k\ge1$, and $(C,z)=\{y^2\pm x^{2k+2}=0\}$, $z=(0,0)$. The only last case in item (4) requires comments. Indeed,
\begin{align}\overline\Sig_\RR(2A_k)=&\big\{y^2\pm(x-t)^{k+1}(x+t)^{k+1}=y^2\pm(x^2-t^2)^{k+1}=0,\nonumber\\ &\qquad t\in(\RR,0)\cup(\RR\sqrt{-1},0)\big\},\label{eA5}\end{align}
and hence $\overline\Sig(2\cdot A_k)$ is regularly parameterised by $\tau=t^2\in(\RR,0)$. The rest of statements follow from the above formula.
\end{proof}

{\medskip\noindent{\bf(3)} {\it Deformations of $E_6$ singularity.}

\begin{lemma}\label{lbs3}
Let a plane curve germ $(C,z)$ be of type $E_6$. Then the locus $\overline\Sig(A_1,2A_2)\subset \Vers(C,z)$ is one-dimensional and smooth. If $(C,z)$ is real, then one component of $\Sig_\RR(A_1,2A_2)$ parameterises curves with a hyperbolic node $A_1^h$ and two real cusps, while the other one - curves with an elliptic node $A_1^e$ and two complex conjugate cusps.
\end{lemma}

\begin{proof}
Note that $\overline\Sig(A_1,2A_2)$ is the equiclassical locus in $\Vers(C,z)$, since $$\begin{cases}\varkappa(A_1)+2\varkappa(A_2)=8=\varkappa(E_6),&\\
\delta(A_1)+2\delta(A_2)=3=\delta(E_6),&\end{cases}$$ (see details in \cite{DH}). Then, by \cite[Theorem 27]{Di}, $\overline\Sig(A_1,2A_2)$ is nonempty and smooth. Its dimension is $\mu(E_6)-\mu(A_1)-2\mu(A_2)=1$.

To recover the singularities of germs parameterized by $\Sig_\RR(A_1,2A_2)$, we only need to exhibit a small real deformation
\begin{equation}y^3+x^4+\sum_{3i+4j<12}a_{ij}(t)x^iy^j=0,\quad a_{ij}(0)=0,\ 0\le3i+4j<12,\quad t\in(\RR,0),\label{ebs1new}\end{equation}
of the quartic $y^3+x^4=0$, which possesses the singularity types as asserted in the lemma, one collection for $t>0$ and the other collection for $t<0$. To this end, consider the deformation
$$y^3+(x^2-ty^2)^2=0,\quad t\in(\RR,0).$$
These real quartics have singularity $E_6$ at the origin and the infinite line as the double tangent; furthermore,
\begin{itemize}\item for $t>0$: the double tangent touches the quartic at two real point, and the quartic has two real flexes,
\item for $t<0$: the double tangent touches the quartic at two complex conjugate points, and the quartic has two complex conjugate flexes.
\end{itemize}
Passing to the dual curves, we obtain the desired deformation (\ref{ebs1new}).
\end{proof}

\begin{lemma}\label{lbs3a}
Let $(C,z)$ be a plane curve singularity of type $E_6$. Then the locus $\overline\Sig(A_5)\subset \Vers(C,z)$ is one-dimensional and smooth. If $(C,z)$ is real, then one component of $\Sig_\RR(A_5)$ parameterises curves with real singularity $A_5^h$, while the other one - curves with real singularity $A_5^e$.
\end{lemma}

\begin{proof}
Let $(C,z)=\{y^3+x^4=0\}$, $z=(0,0)$. Then $\Vers(C,z)$ is isomorphic to the germ at $0$ of the space
$\Span_\CC\{1,x,x^2,y,xy,x^2y\}$.
We can describe $\overline\Sig(A_5\vert E_6)$ as follows. Consider the (extended) versal deformation with the base
$$\Vers'(C,z)=\Span_\CC\{1,x,x^2,y,xy,x^2y,y^2\}.$$ The locus $\overline\Sig'(A_5)\subset \Vers'(C,p)$ is given by
$$\{(y+s)^3+x^4+2tx^2(y+s)+t^2(y+s)^2=0,\ s,t\in (\CC,0)\}$$
Substituting $y-s-t^2/3$ for $y$, we obtain the family
$$\overline\Sig(A_5)=\left\{y^3+x^4+tx^2y+\frac{t^4}{3}y-\frac{t^6}{27}-\frac{t^2}{3}x^3=0,\ t\in(\CC,0)\right\}
\subset\Vers(C,z),$$
which is one-dimensional and smooth, since it is regularly parameterized by $t$.
If $(C,z)$ is real, then a direct computation shows that, for $t>0$, the above family exhibits germs with singularity $A_5^h$, while for $t<0$, - germs with singularity $A_5^e$.
\end{proof}

\subsection{Wall-crossing: Transition of a singularity through a fixed point}\label{secbs9}
Let $V_d(S,S_1,...,S_r)$ ($r\ge0$) be a $T$-smooth equisingular family of dimension $n\ge3$. Let $C\in V_d(S,S_1,...,S_r)$ and let $p\in\Sing(C)$ be the singular point of type $S$. Suppose that the natural projection $\rho:(V_d(S,S_1,...,S_r),C)\to(\PP^2,p)$, taking a curve $C'$ in the germ of $V_d(S,S_1,...,S_r)$ at $C$ to the singular point $p'\in\Sing(C')\cap B(C,p)$ of type $S$, is a submersion. In such a case, the fiber $V_d(S^{fix},S_1,...,S_r):=\rho^{-1}(p)$ is $T$-smooth of dimension $n-2$. Introduce also the locus
$$V_{d,p}(S,S_1,...,S_r)=\{C'\in V_d(S,S_1,...,S_r)\ :\ p\in C'\}.$$
Its dimension equals $n-1$.

\begin{lemma}\label{lbs6}
Accepting the notation and the assumptions of the preceding paragraph, we additionally suppose that $S=A_2$.

(1) Then
$$(V_{d,p}(A_2,S_1,...,S_r),C)\simeq(V_d(A_2^{fix},S_1,...,S_r),C)\times(\{\alpha^2+\beta^3=0\},(0,0)).$$

(2) Suppose that $C$ and $p$ are real. Then, for any partition $n=n_{re}+2n_{im}$, $n_{re}>0$, there exists a configuration $\bp\in{\mathcal P}_{n_{re},n_{im}}$ of $n$ distinct points on $C$ and a path $$\bp(t)\in{\mathcal P}_{n_{re},n_{im}},\ -\eps<t<\eps,\quad\text{where}\quad\bp(0)=\bp,\quad p\in\bp(t)\ \text{for all}\ t,$$
such that
the intersection of the set
$V_d^\RR(S,S_1,...,S_r,\bp(t))$ with a small neighborhood of $C$ in $\vert{\mathcal O}_{\PP^2}(d)\vert$
consists of two elements which are real curves $\RR$-isosingular to $C$ as $0<t<\eps$, or is empty as $-\eps<t<0$.
\end{lemma}

\begin{proof}
(1) The geometric meaning of the first part of the lemma is that $V_{d,p}(S,S_1,...,S_r)$ has the singular locus  $V_d(S^{fix},S_1,...,S_r)$ of codimension one, and the transverse sections are curve germs with singularity $A_2$.

Pick $(n-2)$ smooth points $p_1,...,p_{n-2}\in C$ in general position. Then the linear system $\Lambda_d(p_1,...,p_{n-2})$ (defined as in (\ref{Lambda})) intersects transversally with the germ of $V_d(S^{fix},S_1,...,S_r)$ at $C$, and the intersection consists of the unique element $\{C\}$. Through each point $p_i$, $1\le i\le n-2$, draw a line $L_i\subset\PP^2$ transversally intersecting $C$, and choose a regular coordinate $t_i\in(\CC,0)$ on the germ $(L_i,p_i)$. Then the tuple $(t_1,...,t_{n-2})$ defines regular coordinates on $(V_d(S^{fix},S_1,...,S_r),C)$ via the isomorphism
$$(t_1,...,t_{n-2})\in(\CC^{n-2},0)\mapsto (V_d(S^{fix},S_1,...,S_r),C)\cap\Lambda_d(p_1(t_1),...,p_{n-2}(t_{n-2})).$$
Let $F(x,y,t_1,...,t_{n-2})=0$ be an affine equation of the corresponding element $C(t_1,...,t_{n-2})\in(V_d(S^{fix},S_1,...,S_r),C)$, where $p=(0,0)$. It is easy to see that the formula $F(x+\alpha,y+\beta,t_1,...,t_{n-2})=0$ describes all elements of the germ of $V_d(S,S_1,...,S_r)$ at $C$, and $(\alpha,\beta,t_1,...,t_{n-2}\in(\CC^n,0)$ are regular coordinates on the latter germ (see, for instance, \cite[Lemma 1]{Sh2}). In these coordinates, the intersection $V_{d,p}(S,S_1,...,S_r)\cap\Lambda(p_1,...,p_{n-2})$ is given by
\begin{equation}\beta^2+\alpha^3+\text{h.o.t.}=0.
\label{e-tran}\end{equation}
This proves the first statement of the lemma.

(2) The result of the preceding item implies that the germ of $V_{d,p}(S,S_1,...,S_r)$ at $C$ possesses the tangent cone which is the linear space of dimension $n-1=\dim V_{d,p}(S,S_1,...,S_r)$, namely, the direct sum of the tangent space $T_CV_d(S^{fix},S_1,...,S_r)$ and the tangent line to the cuspidal curve germ $V_{d,p}(S,S_1,...,S_r)\cap\Lambda_d(p_1,...,p_{n-2})$. Pick one more point $p_{n-1}\in C\setminus\{p_1,...,p_{n-2}\}$ in general position. We claim that the linear subsystem $\Lambda_d(p_1,...,p_{n-1})\subset\Lambda_d(p_1,...,p_{n-2})$ intersects $\TCon_CV_{d,p}(S,S_1,...,S_r)$ transversally. To this end, we shall verify that $\Lambda_d(p_1,...,p_{n-1})$ intersects $V_{d,p}(S,S_1,...,S_r)$ at $C$ with multiplicity $2$.
$$(y+\beta)^2+(x+\alpha)^3=0,\quad (\alpha,\beta)\in(\CC^2,0).$$
We can choose the point $p_{n-1}\in C$ very close to $p$ so that $_{n-1}=(x_0,y_0)$, where $y^2_0+x_0^3=0$. We can describe the germ of $V^\RR_{d,p}(S,S_1,...,S_r)$ at $C$ as the one-parameter family of curves
$$(y+\tau^3)^2+(x-\tau^2)^3=0,\quad\tau\in(\CC,0).$$
Slightly deforming the linear system $\Lambda_d(p,p_1,...,p_{n-1})$ by moving the point $p_{n-1}$ along the line $x=x_0$, we obtain the considered intersection in the form
$$(y+\tau^3)^2+(x_0-\tau^2)^3=0\quad\Longleftrightarrow\quad y^2+x_0^3=3x_0^2\tau^2+O(\tau^3)$$
which means that, for $y$ close to $y_0$, one obtains two simple roots for $\tau\in(\CC,0)$.

If $C$ and $p$ are real, the transversality of intersection proved in the preceding paragraph descents to the case of $\{p_1,...,p_{n-1}\}\in{\mathcal P}_{n_{re}-1,n_{im}}$ lying on $C$ in general position. Combining this with the result of item (1), we prove the existence of a path $\{p_1(t),...,p_{n-1}(t)\}\in{\mathcal P}_{n_{re}-1,n_{im}}$ realizing the wall-crossing in item (2i).
\end{proof}

\subsection{Wall-crossing: Non-transversality of the point constraints}\label{secbs_trans}

Here we answer the following question: Given a $T$-smooth variety $V^{irr}_d(S_1,...,S_r)$ of dimension $n$, a curve $C\in V^{irr}_d(S_1,...,S_r)$, and a configuration of $n$ distinct points $\bp\subset C\setminus\Sing(C)$, does the linear system $\Lambda_d(\bp)$ intersect transversally with the tangent space $T_CV^{irr}_d(S_1,...,S_n)$?

Our answer is: \begin{itemize}\item the transversality always holds when $S_1,...,S_n\in\{A_1,A_2\}$, \item the transversality fails whenever there is $S_i\not\in\{A_1,A_2\}$ and $\bp$ belongs to a certain hypersurface in $\Sym^n(C)$.\end{itemize}

\begin{lemma}\label{lbs-trans1}
Consider the locus $V^{irr}_d(kA_1,lA_2)$, where $k+l=\frac{(d-1)(d-2)}{2}$. Then

(1) The locus $V^{irr}_d(kA_1,lA_2)$ is non-empty iff $l\le\frac{3}{2}d-3$, and moreover, under the latter condition, it is $T$-smooth of dimension $n=\frac{d(d+3)}{2}-k-2l$.

(2) For any curve $C\in V^{irr}_d(kA_1,lA_2)$ and any collection of $n$ distinct points $p_1,...,p_n\in C\setminus\Sing(C)$, the tangent space $T_CV_d(kA_1,lA_2)$ intersects transversally with the linear system $\Lambda_d(p_1,...,p_n)$.
\end{lemma}

\begin{proof} We exclude the trivial cases of $d\le 2$ and assume that $d\ge3$.

{\bf(1)} The Pl\"ucker formula (intersection of the curve with its Hessian) yields $6k+8l\le3d(d-2)$. Together with our assumption $k+l==\frac{(d-1)(d-2)}{2}$ this implies the required upper bound
\begin{equation}l\le\frac{3}{2}d-3.\label{e-32d}\end{equation}
Since $\frac{3}{2}d-3<3d$, Lemma \ref{TS}(1) implies that the family $V^{irr}_d(kA_1,lA_2)$ is $T$-smooth of dimension $n$ or empty.

Let us show that the bound (\ref{e-32d}) is sharp. More precisely, for any $0\le l\le\frac{3}{2}d-3$, there exists a rational nodal-cuspidal curve of degree $d$ with $l$ cusps.
To start with, we observe that if there exists a rational nodal-cuspidal curve of degree $d$ with $l=\left[\frac{3}{2}d-3\right]$ cusps, then there exists a rational nodal-cuspidal curve of degree $d$ with any smaller number $l'<l$ of cusps. This follows from Lemmas \ref{TS} and \ref{TS2} which imply that one can convert prescribed cusps into nodes and keep the remaining nodes and cusps. The next observation is that, for an even degree $d=2m$, a rational nodal-cuspidal curve of degree $d$ with $\frac{3}{2}d-3=3m-3$ cusps is just the dual curve of a generic nodal rational curve of degree $m+1$, whereas, for an odd degree $d=2m+1\ge5$,
a rational nodal-cuspidal curve of degree $d$ with $\left[\frac{3}{2}d-3\right]=3m-2$ cusps is the dual curve to a rational curve of degree $m+2$ with nodes and one cusp. The latter curve can be viewed as a rational nodal curve with Newton polygon
$$\conv\{(0,2),(3,0),(m+2,0),(0,m+2)\},$$
see Lemma \ref{lbs7}.

{\bf(2)} 
As we notices in Section \ref{sec-sing}(3), the tangent space $T_CV^{irr}_d(kA_1,lA_2)$ can be identified with
$$\left\{h\in H^0(C,{\mathcal O}_C(d)\ :\ h(z)=0,\ z\in N\quad\text{and}\quad \ord h\big\vert_{(C,z)}\ge3,\ z\in K\right\},$$
where $N$ and $K$ are the sets of the nodes and cusps of $C$, respectively.
Then the tranversality asserted in the lemma can be restated in terms of the normalization $\nu:\PP^1\to C$ as follows:
$$H^0\left(\PP^1,\nu^*{\mathcal O}_C(d)\otimes{\mathcal O}_{\PP^1}(-\nu^{-1}(N)-3\nu^{-1}(K)-\sum_{i=1}^n\nu^{-1}(p_i))\right)=0.$$
Since
$$\deg\left(\nu^*{\mathcal O}_C(d)\otimes{\mathcal O}_{\PP^1}(-\nu^{-1}(N)-3\nu^{-1}(K)-\sum_{i=1}^n\nu^{-1}(p_i))\right)$$
$$=d^2-2k-3l-n=-1>-2,$$ the Riemann-Roch theorem yields the required $H^0$-vanishing.
\end{proof}

\begin{lemma}\label{lbs4}
Let 
$V^{irr}_d(S_1,...,S_r)$ 
be $T$-smooth, $n=\dim V^{irr}_d(S_1,...,S_r)$. Then the set
$$P:=\big\{\bp\in\Sym^n(\PP^2)\ :\ \dim V_d(S_1,...,S_r,\bp)>0\big\}$$
has codimension $\ge2$ in $\Sym^n(\PP^2)$.
\end{lemma}

\begin{proof}
Let $V\subset V^{irr}_d(S_1,...,S_r)$ be an irreducible component (of dimension $n$). Consider he incidence variety
$$Y=\{(C,\bp)\in V\times\Sym^n(\PP^2)\ :\ \bp\subset C\}.$$
The projection of $Y$ onto $V$ has irreducible $n$-dimensional fibers, and hence $Y$ is irreducible of dimension $2n$. On the other hand, over a Zariski open, dense subset of $\Sym^n(\PP^2)$, the projection $Y\to\Sym^n(\PP^2)$ is finite. Assuming that there exists a component $P'\subset P$ of dimension $2n-1$, we would obtain an additional (full-dimensional) component of $Y$ contrary to the irreducibility of $Y$.
\end{proof}

\begin{lemma}\label{lbs2}
For any topological singularity type $S$, the following relation holds
\begin{equation}c(S)\ge\varkappa(S)-\delta(S)\label{ebs2}\end{equation}
with the equality only when $S=A_1$ or $A_2$.
\end{lemma}

\begin{proof}
For $A_1$ and $A_2$ the equality in (\ref{ebs2}) follows from the formulas in Section \ref{sec-sing}(2,3). For an arbitrary plane curve germ $(C,z)=\{f(x,y)=0\}$ with an isolated singular point $z$, denote by $C_1,...,C_s$ all local branches.
The tangent space to the equisingular locus $ES(C,z)$ is contained in the ideal (see \cite[Inequality (19)]{GuS})
\begin{equation}I'_z=\left\{h\in\cO_{C,z}/J(f)\ \Big\vert\ \begin{array}{lcr} \ord h\big\vert_{C_i}\ge2\delta(C_i)+\sum_{j\ne i}(C_i\cdot C_j)_z+\mt(C_i,z)-1\\ \qquad\text{for all}\ i=1,...,s\end{array}\right\}.\label{e-array1}\end{equation}
Since $\dim\cO_{C,z}/I'_z\ge\varkappa(C,z)-\delta(C,z)$ (see \cite[Lemma 10]{GuS}), the inequality (\ref{ebs2}) follows. On the other hand, the tangent space to the locus $ES^{fix}(C,z)\subset ES(C,z)$, which parameterizes curve germs with singularity at $z$, is contained in the ideal (see \cite[Inequality (20)]{GuS})
$$I''_z=\left\{h\in\cO_{C,z}/J(f)\ \Big\vert\ \begin{array}{lcr}\ord h\big\vert_{C_i}\ge2\delta(C_i)+\sum_{j\ne i}(C_i\cdot C_j)_z+2\mt(C_i,p)-1\\ \qquad\text{for all}\ i=1,...,s\end{array}\right\},$$
while $\dim\cO_{C,z}/I''_z=\varkappa(C,z)+\mt(C,z)-\delta(C,z)$ (see \cite[Theorem 3 and Lemma 10]{GuS}). Thus, $c(C,z)+2\ge\varkappa(C,z)+\mt(C,z)-\delta(C,z)$ with a possible equality only if $\mt(C,z)=2$, i.e., when $(C,z)$ is of type $A_k$. However, in such a case, $c(A_k)=k$, $\varkappa(A_k)=k+1$, $\delta(A_k)=\left[\frac{k+1}{2}\right]$, and hence $c(A_k)>\varkappa(A_k)-\delta(A_k)$ for all $k\ge3$.
\end{proof}

\begin{lemma}\label{lbs8}
Let $V^{irr}_d(S_1,...,S_r)$ be $T$-smooth of dimension $n$
and let
\begin{itemize}\item $\sum_{i=1}^r\delta(S_i)=\frac{(d-1)(d-2)}{2}$ (i.e., all curves in the family $V_d(S_1,...,S_r)$ are rational),
\item the sequence $S_1,...,S_r$ contains a singularity type different from $A_1$ and $A_2$.
\end{itemize} Then \begin{enumerate}\item[(i)]
There exists a hypersurface $N_n(d,S_1,...,S_r)\subset\Sym^n(\PP^2)$ parameterizing configurations $\bp$ of $n$ distinct points in $\PP^2$
    such that \begin{itemize}\item $V^{irr}_d(S_1,...,S_r,\bp)\ne\emptyset$,
    \item for some curve $C\in V^{irr}_d(S_1,...,S_r,\bp)$,
    $$h^0(C,{\mathcal J}_{Z^{es}/C}(d)\otimes\cO_C(-\bp))>0.$$
\end{itemize}
\item[(ii)] There exists a non-empty Zariski open subset $N_n^0(d,S_1,...,S_r)\subset N_n(d,S_1,...,S_r)$ such that
\begin{itemize}\item $N_n^0(d,S_1,...,S_r)\cap {\mathcal P}_{n_{re},n_{im}}\ne\emptyset$ for all partitions $n_{re}+2n_{im}=n$ with $n_{re}>0$;
\item for each $\bp\in N_n^0(d,S_1,...,S_r)$, there exists
$C\in V^{irr}_d(S_1,...,S_r,\bp)$ such that the germ of $V^{irr}_d(S_1,...,S_r)$ at $C$ intersects with the linear system $\Lambda_d(\bp)$ at $C$ with multiplicity $2$, whereas any other germ of $V^{irr}_d(S_1,...,S_r)$ at $C'\in V^{irr}_d(S_1,...,S_r,\bp)\setminus\{C\}$, intersects with $\Lambda_d(\bp)$ transversally, i.e.,
$$h^0(C',{\mathcal J}_{Z^{es}(C')/C'}(d)\otimes\cO_{C'}(-\bp))=0.$$
\end{itemize}
\end{enumerate}
\end{lemma}

\begin{proof} Observe that $d\ge4$ since a rational cubic has either one node or one cusp. Note also that $n\ge4$.  Indeed, any curve $C\in V^{irr}_d(S_1,...,S_r)$ can be included into one-parameter deformation of the union of $d$ distinct lines through one point (denote by $S$ the topological type of the latter singular point). It follows that
$$c(S_1)+...+c(S_r)<c(S)=\frac{d(d+1)}{2}-2\quad\Longrightarrow\quad n>\frac{d(d+3)}{2}-\frac{d(d+1)}{2}+2=d+2.$$

{\bf(i)} Pick a curve $C\in V^{irr}_d(S_1,...,S_r)$. For each singular point $z\in\Sing(C)$ consider the zero-dimensional scheme $Z'_z\subset C$ concentrated at $z$ and defined by the ideal $I'_z\subset\cO_{C,z}$ (see the proof of Lemma \ref{lbs2}). As noticed in the proof of Lemma \ref{lbs2}, $I'_z\supset I^{es}(C,z)$ (see \cite[Inequality (19)]{GuS} and \cite[Page 435, item (ii)]{DH}). The family $V^{irr}_d(S_1,...,S_r)$ is $T$-smooth, that is (see \cite[Theorem 2.2.55]{GLS}),
$$h^1(C,{\mathcal J}_{Z^{es}(C)/C}(d))=0,$$ and hence
\begin{equation}h^1(C,{\mathcal J}_{Z'(C)/C}(d))=0,\label{eh1-1}\end{equation} where the zero-dimensional scheme $Z'(C)$ is supported at $\Sing(C)$ and defined by the ideals $I'_z$, $z\in\Sing(C)$. Since $I'_z$ is contained in the conductor ideal $I^{cond}(C,z)$ (see, for instance, \cite[Section 1.1.2.6]{GLS}), we get ${\mathcal J}_{Z'(C)/C}(d)=\pi_*\cO_{\PP^1}(\Delta)$, where $\pi:\PP^1\to C$ is the normalization and
$$\deg\Delta=d^2-(d-1)(d-2)-\sum_{i=1}^r(\varkappa(S_i)-\delta(S_i))=\frac{d(d+3)}{2}-1-
\sum_{i=1}^r(\varkappa(S_i)-\delta(S_i)).$$
By Riemann-Roch and (\ref{eh1-1})
$$h^0(C,{\mathcal J}_{Z'(C)/C}(d))=\frac{d(d+3)}{2}-\sum_{i=1}^r(\varkappa(S_i)-\delta(S_i))\overset{\text{def}}{=}n_0+1.$$
As shown in the proof of Lemma \ref{lbs2},
$$H^0(C,{\mathcal J}_{Z^{es}(C)/C}(d))\subset H^0(C,{\mathcal J}_{Z'(C)/C}(d)),$$
and
$$n\overset{\text{def}}{=}h^0(C,{\mathcal J}_{Z^{es}(C)/C}(d))=\frac{d(d+3)}{2}-\sum_{i=1}^rc(S_i)\le n_0.$$

Note that
$$n_0=d^2-\sum_{i=1}^r\varkappa(S_i).$$ It means that every non-trivial section $h\in H^0(C,{\mathcal J}_{Z'(C)/C}(d))$ vanishes with multiplicity $\varkappa(C,z)$ at each singular point $z\in\Sing(C)$ (more precisely, vanishes with multiplicity $2\delta(Q)+\sum_{Q'\ne Q}(Q\cdot Q')_p+\mt(Q,p)-1$ on each local branch $Q$ of the germ $(C,z)$ (cf. (\ref{e-array1})) and vanishes at an extra divisor $\Delta_{n_0}$ of degree $n_0$. Note that, for almost all sections, $\Delta_{n_0}$ consists of $n_0$ distinct points disjoint from $\Sing(C)$: to this end, it is enough to exhibit just one section like that, and we can take the section represented by the product of a generic first derivative of the polynomial defining $C$ and a generic linear form. Thus, we obtain a morphism
$$\vert {\mathcal J}_{Z'(C)/C}(d)\vert \to\Sym^{n_0}(C).$$
It follows from B\'ezout's theorem, that this morphism is (generically) injective, and hence its image defines a hypersurface in $\Sym^{n_0}(C)$. As a consequence, the multivalued morphism
$$\vert {\mathcal J}_{Z^{es}(C)/C}(d)\vert \to\Sym^n(C),$$ which takes a section of $H^0(C,{\mathcal J}_{Z'(C)/C}(d))$ to the bunch of subdivisors of $\Delta_{n_0}$ of degree $n$, has an image of dimension $n-1$.
This, we obtain a hypersurface $\Pi$ in $\Sym^n{\mathcal C}$, where ${\mathcal C}\to V^{irr}_d(S_1,...S_r)$ is the universal curve. By Lemma \ref{lbs4}, the natural projection $\Pi\to\Sym^n(\PP^2)$ restricted to an open dense subset of $\Pi$ is a finite-to-one morphism; hence, the image is a hypersurface $N_n(d,S_1,...,S_r)\subset\Sym^n(\PP^2)$.

{\bf(ii)} Now, pick a generic real curve $C\in V^{irr,\RR}_d(S_1,...,S_r)$ and a generic conjugation-invariant configuration of $n-1$ points on $C$. Then this configuration can be completed with a real point of $C$ to an element of $\Pi$. This proves the first statement in item (ii) of the lemma.

To prove the second statement in item (ii) of the lemma, take a generic $C\in V^{irr}_d(S_1,...,S_r)$ and a generic section $h\in H^0(C,{\mathcal J}_{Z^{es}(C)/C}(d))$ which vanishes at each singular point $z\in\Sing(C)$ with (the total) multiplicity $\varkappa(C,z)$ as well as in $n_0$ distinct smooth points. Pick a subset $\bp$ of $n$ points among the latter $n_0$ points, and choose a point $p\in\bp$. Varying $h\in H^0(C,{\mathcal J}_{Z^{es}(C)/C}(d))$, we can suppose that $p$ is not a flex. Removing the point $p$ from the constraints, we obtain a germ at $C$ of a one-dimensional equisingular family $F$ of curves passing through $\bp\setminus\{p\}$. Consider the enveloping curve $E(C,\bp,p)$ of this family in a neighborhood of $p$. By the generality of choices made, $E$ is smooth and, moreover, $E$ contains $p$, since
\begin{align}T_CF&=H^0(C,{\mathcal J}_{Z^{es}(C)/C}(d)\otimes\cO_C(-(\bp\setminus\{p\})))\nonumber\\
&=H^0(C,{\mathcal J}_{Z^{es}(C)/C}(d)\otimes\cO_C(-\bp)).\label{e-envelope}\end{align}
We claim that any curve $C'\in F$ intersects with $C$ in a neighborhood of $p$ transversally at one point. Indeed, due to the generality of the choices made, a non-zero element of $T_CF=H^0(C,{\mathcal J}_{Z^{es}(C)/C}(d)\otimes\cO_C(-\bp))$ vanishes with the total order $\sum_{i=1}^r\varkappa(S_i)$ at $\Sing(C)$ as well as at $\bp$ and at some extra $n_0-n$ points with order $1$, and hence the claim follows. Furthermore, this claim yields that the curves of $F$ are quadratically tangent to $E$ in a neighborhood of $p$. Thus, slightly moving $p$ in a direction transversal to $E$ to another point $p'$, we obtain exactly two curves $C',C''\in F$ passing through $p'$ and through $\bp\setminus\{p'\}$, which proves that the intersection multiplicity of the germ $(V^{irr}_d(S_1,...,S_r),C)$ with the linear system $\Lambda_d(\bp)$ equals $2$.

At last, we prove the transversality of the intersection of the germ $(V^{irr}_d(S_1,...,S_r),C')$ with $\Lambda_d(\bp)$ for each curve $C'\in V^{irr}_d(S_1,...,S_r),\bp)\setminus\{C\}$. We show this by a suitable small deformation of the pair $(C,\bp)$. Indeed, assume that $C'\in V^{irr}_d(S_1,...,S_r)\setminus\{C\}$ is such that $\bp$ consists of zeros of some non-trivial section $g'\in H^0(C',{\mathcal J}_{Z^{es}(C')/C'}(d))$. Now we release the point $p\in\bp$ while keeping $\bp\setminus\{p\}$ fixed, and obtain the germ $F$ at $C$ of the (global) one-dimensional family $V^{irr}_d(S_1,...,S_r,\bp\setminus\{p\})$ and its enveloping curve $E(C,\bp,p)$ as well as the germ $F'$ at $C'$ of the same equisingular family $V^{irr}_d(S_1,...,S_r,\bp\setminus\{p\})$ and its enveloping curve $E(C',\bp,p)$.
Using relation (\ref{e-envelope}), we make the following observation: if $\widetilde C\in F$ and a point $\widetilde p$ in a neighborhood of $p$ is such that $$h^0(\widetilde C,{\mathcal J}_{Z^{es}(\widetilde C)/\widetilde C}(d)\otimes\cO_{\widetilde C}(-\bp+p-\widetilde p))>0,$$ then $\widetilde p$ is the tangency point of $\widetilde C$ and $E(C,\bp,p)$. The same holds for the family $F'$. Thus, if we move $C$ along $F$ and $p$ along $E(C,\bp,p)$, then $p$ leaves the enveloping curve $E(C',\bp,p)\ne E(C,\bp,p)$, and $C'$ turns into a curve $\widetilde C'$ for which $$h^0(\widetilde C',{\mathcal J}_{Z^{es}(\widetilde C')/\widetilde C'}(d)\otimes\cO_{\widetilde C}(-\bp+p-\widetilde p))=0,$$ as required.

It remains to explain that $E(C',\bp,p)\ne E(C,\bp,p)$ in the preceding paragraph. Clearly, the inequality holds if $C$ and $C'$ have distinct tangents at $p$. We claim that the relation $T_qC=T_qC'$ holds in at most one point of $\bp$: indeed, due to $n\ge4$, the position and the tangents to $C$ at two points can be chosen in a generic way; since the remaining $n-2$ points of $\bp$ vary in an $(n-3)$-dimensional family, a curve $C'$ passing through $\bp$ and tangent to $C$ at two points would vary in an $(n+1)$-dimensional family contrary to the upper bound $n$. That is, we always can choose a point $p\in\bp$ at which $T_pC\ne T_pC'$.
\end{proof}

\begin{remark}
The use of the enveloping curve in the end of the proof of Lemma \ref{lbs8} resembles the proof of \cite[Theorem 3.1]{IKS0} stating the lack of invariance of Welschinger's signed count of real plane curves of positive genera.
\end{remark}

\section{Enumeration of real cuspidal cubics}
\label{secbs2}

The general Welschinger result \cite[Theorem 0.1]{Wel1} implies that the enumeration of real plane cuspidal cubics passing through a generic conjugation-invariant configuration of $7$ points is not invariant. Here, we comment on the two wall-crossing events, in which the number of real cuspidal cubics jumps or drops by $2$.

\medskip{\bf(a)} One of these events is the transition through the wall that parameterises configurations of $7$ points through which one can trace a real cubic with $A_3$ singularity. Such a cubic $C$ is the union of a line $L$ and a conic $Q$ tangent to the line at some point $z$. We consider a particular place on this wall when $5$ generic points stay on the conic $Q$ and $2$ points lie on $L\setminus\{z\}$. It is easy to verify the hypotheses of Lemma \ref{local-global}. Hence, by Lemma \ref{lbs1}(2,4), two real cuspidal cubics appear or disappear (see Figure \ref{fbs1}), that breaks the invariance of the count.

\begin{figure}
\setlength{\unitlength}{1cm}
\begin{picture}(14,4)(-1.3,-0.5)
\includegraphics[width=0.8\textwidth, angle=0]{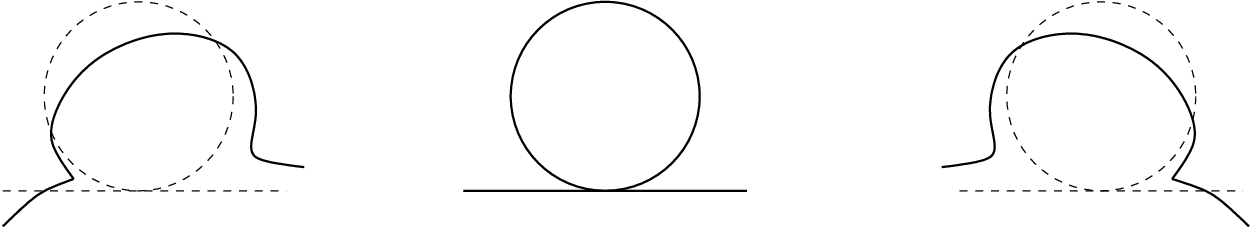}
\put(-6.9,0.9){$\Longleftarrow$}\put(-3.3,0.9){$\Longrightarrow$}
\end{picture}
\caption{Deformation $A_3\Longrightarrow A_2$ (shown by solid line): both real deformations are on the same side of the bifurcation} \label{fbs1}
\end{figure}

\medskip{\bf(b)} The other event is the crossing of the wall that parameterises configurations of $7$ points, one of which appears to be at the cusp. So, let $C$ be a real cuspidal cubic, $z_1$ be the cuspidal singular point, and $z_2,...,z_7$ a generic configuration of $6$ real points on $C$. Then, by Lemma \ref{lbs6}(2i), there exists a path
in ${\mathcal P}_{7,0}$ (with fixed $z_1$) along which two counted real cuspidal cubics appear or disappear breaking the invariance of the count.

\section{Failure of invariance, I: curves with a singularity different from node or cusp}
\label{secbs8new}

\begin{theorem}\label{tbs3new}
Let $V^{irr}_d(S_1,...,S_r)$ be a $T$-smooth locus of dimension $n=n_{re}+2n_{im}$ such that
\begin{itemize}\item $d\ge4$, \item $\sum_{i=1}^r\delta(S_i)=\frac{1}{2}(d-1)(d-2)$ (i.e., all curves $C\in V_d(S_1,...,S_r)$ are rational), \item $n_{re}>0$,
\item the sequence $S_1,...,S_r$ contains a singularity type different from $A_1$ and $A_2$.
    \end{itemize} Then there is no real singularity weight $\varphi:\RR TT\to R\setminus\{0\}$ such that the number
$$\sum_{C\in V^{irre,\RR}_d(S_1,...,S_r;\bp)}w_\varphi(C)$$
is invariant of the choice of a generic configuration $\bp\in{\mathcal P}_{n_{re},n_{im}}$.
\end{theorem}

\begin{proof}We accept the notations in Lemma \ref{lbs8}.

Consider a curve $C\in V_d^{irr,\RR}(S_1,...,S_r)$ and a configuration $\bp\subset C\setminus\Sing(C)$ of $n$ distinct points belonging to $$N^0_n(d,S_1,...,S_r)\cap{\mathcal P}_{n_{re},n_{im}}$$ (which is nonempty by Lemma \ref{lbs8}(ii)). Let $p\in\bp$ be a real point. Let $F$ be the germ of the one-dimensional family of real curves, and let $E$ be its enveloping curve as introduced in the proof of Lemma \ref{lbs8}. Take the germ of a smooth curve $M$ in $\RR P^2$ transversally crossing the enveloping curve $E$ at $p$ and define a path in ${\mathcal P}_{n_{re},n_{im}}$ consisting of the fixed part $\bp\setminus\{p\}$ and a mobile point $p(t)$ moving along $M$ so that $p(0)=p$. Thus, when $t$ changes sign, two real $\RR$-isosingular curves in the family $F$ turn into two complex conjugate curves, see Figure \ref{fbs4}. Hence, the invariance of the count fails in view of the fact that this wall-crossing event does not interfere with any other one (see Lemma \ref{lbs8}(ii)).
\end{proof}

\begin{figure}
\setlength{\unitlength}{1cm}
\begin{picture}(5,5)(-2,0)
\includegraphics[width=0.6\textwidth, angle=0]{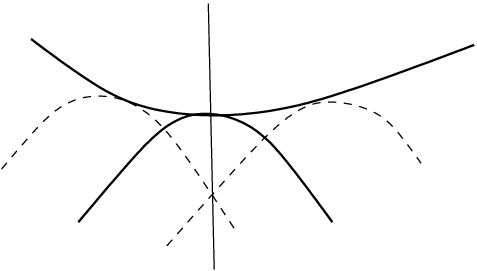}
\put(-7,3.5){$E$}\put(-2,0.5){$C$}
\put(-4.3,2.5){$p$}\put(-4.3,4.1){$M$}
\end{picture}
\caption{Wall-crossing in Theorem \ref{tbs3new}} \label{fbs4}
\end{figure}
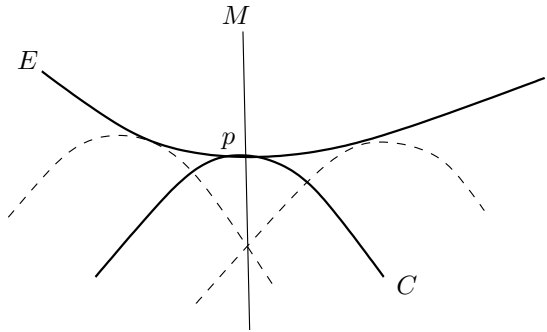

\begin{remark}\label{r-fail1}
The conclusion of Theorem \ref{tbs3new} holds also when $n_{re}=0$, provided, $N_n^0(d,S_1,...,S_r)\cap {\mathcal P}_{0,n_{im}}\ne\emptyset$ (cf. Lemma \ref{lbs8}).
Indeed, let $\bp\in N_n^0(d,S_1,...,S_r)\cap {\mathcal P}_{0,n_{im}}$, and let $C\in V_r^{irr,\RR}(S_1,...,S_n,\bp)$ be such that $V_r^{irr}(S_1,...,S_n,\bp)$ intersects with $\Lambda_d(\bp)$ at $C$ with multiplicity $2$. Taking into account that $V_d^{irr,\RR}(S_1,...,S_r)$ is smooth at $C$, we derive that a variation of $\bp in{\mathcal P}_{0,n_{im}}$ along a generic smooth path yields a deformation of the double intersection at $C$ into two real $\RR$-isosingular curves on one side, and into a pair of complex conjugate curves on the other side.

In fact, this argument works well for $n_{re}>0$. We, however, presented another proof of Theorem \ref{tbs3new} which admits a simple geometric visualization.
\end{remark}

\section{Failure of invariance, II: quartics}\label{secbs11}

Recall that according to \cite[Table 1]{Wal} and \cite[Theorem 2]{Sh1} (see also Lemma \ref{TS1}(1) and \cite[Theorem 4.3.8]{GLS}),
we have:

\begin{lemma}\label{rbs2}
A family $V_4^{irr}(S_1,...,S_r)$ of rational non-nodal quartics is nonempty iff the singularity collection $(S_1,...,S_r)$ belongs to the following list:
\begin{equation}\begin{cases}&(2A_1,A_2),\ (A_1,2A_2),\ (A_1,A_3),\ (A_1,A_4),\ (3A_2),\\ &(A_2,A_3),\ (A_2,A_4),\ (A_5),\ (A_6),\ (D_4),\ (D_5),\ (E_6).\end{cases}\label{e-quart}\end{equation}
Furthermore, any nonempty locus $V_4(S_1,...,S_r)$ is $T$-smooth of dimension $n=14-\sum_{i=1}^r\mu(S_i)$.
\end{lemma}

\begin{theorem}\label{tbs1new}
With a single exception designated in Theorem \ref{t-rev2}, for any collection $S_1,...,S_r$ of singularity types mentioned in (\ref{e-quart})
and any partition $n=n_{re}+2n_{im}$, there is no real singularity weight $\varphi:\RR TT\to R\setminus\{0\}$ such that the number
$$\sum_{C\in V^{\RR}_4(S_1,...,S_r,\bp)}w_\varphi(C)$$
is invariant of the choice of a generic configuration $\bp\in{\mathcal P}_{n_{re},n_{im}}$.
\end{theorem}

\begin{proof} We proceed in three steps considering the failure of the count invariance in the wall-crossing events described in Sections \ref{secbs6}, \ref{secbs9}, and \ref{secbs_trans}, respectively.

\smallskip
{\it Step 1.} Observe that there exist quartics
\begin{align*}&C_1\in V^\RR_4(2A_1,A_3),\ C_2\in V^\RR_4(A_1,A_2,A_3),\ C_3\in V^{irr,\RR}_4(A_5),\\ &C_4\in V_4^\RR(A_1,A_5),\
C_5\in V^\RR_4(A_2,A_5),\ C_6\in V^\RR_4(A_7).\end{align*}
For example, $C_5$ is the union of a cuspidal cubic and its tangent line at the flex point, $C_6$ is the union of two conics intersecting at one point with multiplicity $4$. By Lemma \ref{TS1}(1), the above quartics admit deformations into quartics
\begin{align*}&C'_1\in V^{irr,\RR}_4(2A_1,A_2),\ C'_2\in V^{irr,\RR}_4(A_1,2A_2),\ C'_3\in V^{irr,\RR}_4(A_1,A_3),\\ &C'_4\in V_4^{irr,\RR}(A_1,A_4),\
C'_5\in V^{irr,\RR}_4(A_2,A_4),\ C'_6\in V^{irr,\RR}_4(A_6),\end{align*}
respectively.
That is, in the count of real quartics with singularity collections as in the latter list, we necessarily encounter walls-crossings that realize bifurcations of type
\begin{align*}&(A_2\Longrightarrow A_3),\ (A_2\Longrightarrow A_3),\ (A_1,A_3\Longrightarrow A_5),\\ &(A_4\Longrightarrow A_5),\
(A_4\Longrightarrow A_5), \ (A_6\Longrightarrow A_7),\end{align*}
respectively (see Section \ref{secbs6}). Hence, by Lemmas \ref{local-global} and \ref{lbs1}(4), in each of these wall-crossings a pair of real $\RR$-isosingular quartics appears or disappears breaking the invariance of the count.

\smallskip

{\it Step 2.}
Note that any curve $C\in V_4^{irr,\RR}(3A_2)$ has a real cusp. Hence, under the assumption $n_{re}>0$, we necessarily encounter a collision of the real cusp and of a real point of a configuration $\bp\in{\mathcal P}_{n_{re},n_{im}}$ (see Section \ref{secbs9}). Then Lemma \ref{lbs6}(2) implies the failure of invariance of the count with any real singularity weight $\varphi$.

\smallskip{\it Step 3.}
The failure of invariance of the count of curves in $V_4^{irr,\RR}(A_5)$ and in $V_4^{irr,\RR}(D_5)$ follows from Theorem \ref{tbs3new} since in both the cases we have $n=14-5=9$, and hence $n_{re}>0$. The same conclusion holds for the singularity collections $\{D_4\}$ and $\{E_6\}$ as long as $n_{re}>0$. Furthermore, in the two latter cases, the failure of the invariance follows from Remark \ref{r-fail1} since
\begin{equation}N^0_{10}(4,D_4)\cap{\mathcal P}_{0,5}\ne\emptyset\quad\text{and}\quad N^0_8(4,E_6)\cap{\mathcal P}_{0,4}\ne\emptyset\label{e-nonreal}\end{equation}
(here we use the fact that $\dim V_4(D_4)=10$ and $\dim V_4(E_6)=8$).

For the former relation in (\ref{e-nonreal}), consider the quartic
$$C=\{(x^2+y^2-2x)(x^2+y^2)+\eps x^4=0\},\quad 0<\eps\ll1.$$
Its real point set consists of a convex oval close to the circle $x^2+y^2-2x=0$, which passes through the singular point $(0,0)$ of type $D_4$. Pick a point $q$ inside the open disk bounded by this oval and take the section $h\in H^0(C,{\mathcal J}_{Z^{es}(C)/C}(4))$ induced by the curve $(x+1)C_{(q)}$, where $C_{(q)}$ is the $q$-polar curve of $C$. This, the zeros of $h$ are: $(0,0)$ of order $\varkappa(D_4)=6$ and $16-6=10$ other zeros of order $1$ located at $\{x+1=0\}\cap C$ and at the tangency points of $C$ and the lines through $q$. Clearly, all these $10$ points are non-real. For the latter relation in (\ref{e-nonreal}), we start with the quartic
$$C=\{y^3+y^4+x^4=0\}$$ and proceed as in the case of singularity $D_4$.
\end{proof}

\section{Failure of invariance, III: nodal-cuspidal curves}
\label{secbs12}

\begin{theorem}\label{tbs2new}
(1) For any $d\ge5$, each nonempty locus $V^{irr}_d(kA_1,lA_2)$, where $k+l=\frac{(d-1)(d-2)}{2}$, is $T$-smooth of dimension
$n=\frac{d(d+3)}{2}-k-2l$.

(2) Furthermore, if $l>0$, then for any partition $n=n_{re}+2n_{im}$,
there is no real singularity weight $\varphi:\RR TT\to R\setminus\{0\}$
such that the number
$$\sum_{C\in V_d^{irr,\RR}(kA_1,lA_2,\bp)}w_\varphi(C)$$
is invariant of the choice of a generic configuration $\bp\in{\mathcal P}(n_{re},n_{im})$.
\end{theorem}

\begin{proof} {\it Step 1.} Recall that by Lemma \ref{lbs-trans1}, $V^{irr}_d(kA_1,lA_2)\ne\emptyset$ iff $l\le\frac{3}{2}d-3$.
Let us prove a stronger claim: for any $0<l\le\frac{3}{2}d-3$, the locus $V^{irr,\RR}_d(kA_1,lA_2)$, where $k+l=\frac{(d-1)(d-2)}{2}$, is nonempty, and for any partition $n=n_{re}+2n_{im}$, there exists a generic configuration $\bp\in{\mathcal P}_{n_{re},n_{im}}$ and a curve $C_1\in V^{irr,\RR}_d(kA_1,lA_2,\bp)$ having at least one real cusp. The nonemptyness of $V^{irr,\RR}_d(kA_1,lA_2)$ follows by the construction in the proof of Lemma \ref{lbs-trans1}. If $l=1$, then every curve in $V^{irr,\RR}_d(kA_1,A_2)$ has a real cusp. Suppose that $l=\left[\frac{3}{2}d-3\right]\ge2$ and show that there is $C_2\in V^{irr,\RR}_d(kA_1,lA_2)$ with at least two real cusps. Having such a curve and deforming appropriately many pairs of complex conjugate cusps into pairs of complex conjugate nodes and a proper subset of real cusps into real nodes, we can obtain the desired curve with a real cusp for any $0<l'<l$. To obtain such a curve $C_2$, we slightly modify the construction in the proof of Lemma \ref{lbs-trans1}. If $d=2m$, $m\ge3$, we take a generic real nodal rational curve of degree $m+1$ with the Newton polygon
$$\conv\{(0,0),\ (m-2,0),\ (m,1),\ (1,m),\ (0,m-2)\}.$$
It is just a real plane nodal rational curve of degree $m+1$ with at least two real inflexion points, and hence its dual curve $C_2\in V_d^{irr,\RR}(kA_1,lA_2)$ has two real cusps. If $d=4m+1$, $m\ge1$, we take a real nodal rational curve of degree $2m+2$ with the Newton polygon
$$\conv\{(0,2),\ (3,0),\ (2m+1,0),\ (2m-1,3),\ (0,2m+2)\}.$$
Its dual curve $C_2\in V_d^{irr,\RR}(kA_1,lA_2)$ has a real cusp. Since $l=6m-2$, the number of real cusps is even as required. At last, if $d=4m+3$, $m\ge1$, we obtain the required curve $C_2\in V_d^{irr,\RR}(kA_1,lA_2)$ to be the dual to a real rational nodal curve of degree $2m+3$ with Newton polygon
$$\conv\{(0,2),\ (3,0),\ (2m+2,0),\ (2m,3),\ (1,2m+2),\ (0,2m)\}.$$
Finally, we take a configuration $\bp\in{\mathcal P}_{n_{re},n_{im}}$ in general position on $C_2$ and notice that $\bp$ is, in fact, generic in ${\mathcal P}_{n_{re},n_{im}}$ due to the transversality statement in Lemma \ref{lbs-trans1}(2).

\smallskip{\it Step 2.} Suppose that $1\le l\le \frac{3}{2}d-3$, $k+l=\frac{(d-1)(d-2)}{2}$. Then $n=3d-1-l$.

If $3d-1-l$ is odd, we necessarily have $n_{re}>0$ and then, by results of Step 1, we unavoidably encounter the collision of a real point in $\bp\in{\mathcal P}_{n_{re},n_{im}}$ and a real cusp of a curve in $V_d^{irr,\RR}(kA_1,lA_2)$, which by Lemma \ref{lbs6}(2) breaks the invariance of the count for any choice of a real singularity weight $\varphi$. For the same reason, the failure of the invariance occurs when $3d-1-l$ is even and $n_{re}>0$.

\smallskip{\it Step 3.} Let $3d-1-l$ be even and $n_{re}=0$.

If $d$ is odd, then
$$l-1\le\frac{3}{2}d-4\quad\Longrightarrow\quad l-1\le\frac{3}{2}d-\frac{9}{2}=\frac{3}{2}(d-1)-3.$$
Hence there exists a real nodal-cuspidal rational curve $C_{d-1}$ of degree $d-1$ with $l-1$ cusps. Let $L$ be the tangent line to $C_{d-1}$ at a generic real point. Then $C_{d-1}L\in V_d^\RR(kA_1,(l-1)A_2,A_3)$. By Lemma \ref{TS1}(2), the locus $V_d(kA_1,(l-1)A_2,A_3)$ is $T$-smooth at $C_{d-1}L$, and hence (see Lemma \ref{TS2}), $C_{d-1}L$ can be deformed into a curve $C\in V_d^{irr,\RR}(kA_1,lA_2)$. Furthermore, by Lemma \ref{local-global}, we encounter a bifurcation of type $(A_2\Longrightarrow A_3)$, which breaks the invariance of the count with any real singularity weight $\varphi$ (see Lemma \ref{lbs1}(2,4)).

If $d$ is even and $l=1$ (the case considered by Welschinger \cite{Wel1}), we take $C_{d-1}$ to be a real rational nodal curve of degree $d-1$ and proceed as in the preceding paragraph.

Suppose now that $d$ is even and $l\ge2$, and suppose that there exists a real singularity weight inducing an invariant count of curves in $V_d^{irr,\RR}(kA_1,lA_2)$.

(i) Since
$$l-2\le\left(\frac{3}{2}d-3\right)-2<\frac{3}{2}(d-1)-3,$$
there exists a real nodal-cuspidal curve $C_{d-1}$ of degree $d-1$ having $l-2$ cusps and at least one real inflexion point (see the construction in Step 1). Then $C_{d-1}L\in V_d^\RR(kA_1,(l-2)A_2,A_5)$, where $L$ is the tangent line to $C_{d-1}$ at a real inflexion point. By Lemma \ref{TS1}(2), $V_d(kA_1,(l-2)A_2,A_5)$ is $T$-smooth at $C_{d-1}L$, and hence (see Lemma \ref{TS2}) $C_{d-1}L$ can be deformed into a curve $C\in V_d^{irr,\RR}(kA_1,lA_2)$. Furthermore, by Lemma \ref{local-global}, we encounter a bifurcation of type $(2A_2\Longrightarrow A_5)$ which implies (see Lemma \ref{lbs1}(2,4)) that
\begin{equation}\varphi(A_2^{re})^2=\varphi(2A_2^{im})\label{e-inv2}\end{equation}
(here we write $A^{re}_2$ for a real cusp, and $2A_2^{im}$ for a pair of complex conjugate cusps).

(ii) If $l<\frac{3}{2}d-3$, then there exists a curve $C\in V_d^{irr,\RR}((k-1)A_1,(l+2)A_2$ having at least one real cusp (see Step 1). This curve admits a deformation into a curve in $V_d^{irr,\RR}(kA_1,lA_2)$, and hence, we encounter a bifurcation of type $(A_1\Longrightarrow A_2)$, which by Lemma \ref{lbs1}(2,3) yields
\begin{equation}\varphi(A_1^h)=-\varphi(A_1^e).\label{e-inv3}\end{equation}
If $l=\frac{3}{2}d-3$, then we consider a real rational curve $C\in V_d^{irr,\RR}((k-1)A_1,(l-1)A_2,A_4)$. The latter curve can be constructed as the dual one to the real rational curve with the Newton polygon
$$\conv\{(0,2),(4,0),(m+1,0),(0,m+1)\},\quad m=\frac{d}{2},$$
whose lift to the toric surface associated with the given polygon is a generic real nodal rational curve quadratically tangent to the toric divisor corresponding to the edge $[(0,2),(4,0)]$ (cf. Lemma \ref{lbs7}). As above, the curve $C$ can be deformed into a curve in $V_d^{irr,\RR}(kA_1,lA_2)$, and hence we encounter a bifurcation of type $(A_1,A_2\Longrightarrow A_4)$. By Lemma \ref{lbs1}(2,3), the assumed invariance of the count would yield
$$\varphi(A_1^h)\varphi(A^{re}_2)=-\varphi(A_1^e)\varphi(A^{re}_2)$$
equivalent to (\ref{e-inv3}).

(iii) There exists a curve $C\in V_d^{irr,\RR}((k-1)A_1,(l-2)A_2,E_6)$. Indeed, for example we can consider the dual curve to the generic real rational nodal curve with Newton polygon
$$\conv\{(0,1),(4,0),(m+1,0),(0,m+1)\}$$ (cf. Lemma \ref{lbs7}). Combining Lemmas \ref{TS}(1) and \ref{TS2}, we derive that $C$ can be deformed into a curve in $V_d^{irr,\RR}(kA_1,lA_2)$, and hence we encounter a bifurcation of type $(A_1,2A_2\Longrightarrow E_6)$. Then Lemma \ref{lbs3} yields
\begin{equation}\varphi(A_1^e)\varphi(2A_2^{im})=\varphi(A_1^h)\varphi(A_2^{re})^2.\label{e-inv1}\end{equation}

However, the system of equations (\ref{e-inv2}), (\ref{e-inv3}), and (\ref{e-inv1}) has no nonzero solution.

The proof of Theorem \ref{tbs2new} is completed.
\end{proof}

\section{Enumeration of real three-cuspidal quartics}\label{secbs4}

Here we prove Theorem \ref{t-rev2}.

By Lemma \ref{rbs2}, $\dim_\RR V_4^{irr,\RR}(3A_2)=8$. Let $\{\bp(t)\}_{t\in[0,1]}$ be a generic smooth path in ${\mathcal P}_{0,4}$. By Lemma \ref{lbs-trans1}, the number $\#V^{irr,\RR}_4(3A_2,\bp(t))$ remains constant away of the walls parameterizing degenerate curves or quartics $C\in V_4^{irr,\RR}(3A_2,\bp(t))$ such that $\Sing(C)\cap\bp(t)\ne\emptyset$.

Observe that, in our situation, the latter wall does not exist. Indeed, $\dim_\RR{\mathcal P}_{0,4}=16$. If we fix $3$ pairs of generic complex conjugate points, then we obtain a subfamily in $V_4^{irr,\RR}(3A_2)$ of real dimension $2$. Hence each of the singular points of a curve in the latter family moves in $\PP^2$ along a trajectory of real dimension $\le2$, which means that the fourth pair of complex conjugate points is left with at most two real degrees of freedom. Thus, the subfamily of configurations $\bp\in{\mathcal P}_{0,4}$ such that there exists a quartic $C\in V_4^{irr,\RR}(3A_2)$ with $\Sing(C)\cap\bp\ne\emptyset$ has real dimension $\le 3\times 4+2=14$, i.e., has codimension $2$ in ${\mathcal P}_{0,4}$.

By Lemma \ref{rbs2}, a reduced real quartic $C'$, which might be a degeneration of a three-cuspidal quartic and varies in an equisingular family of real dimension $7$, should have one of the following collections of singularities:
$$\{2A_2,A_3\},\ \{A_2,A_5\},\ \{A_2,D_5\},\ \{D_7\},\ \{E_7\}.$$
By \cite[Theorem 4]{Lya}, the last three singularity collections are not possible. The first one is not possible neither, since it corresponds to a reducible quartic with two cusps that does not exist.

In the case of $C'\in V_4^{\RR}(A_2,A_5)$, the curve $C'$ must be the union of a real cuspidal cubic $C_3$ and the tangent line $L$ to $C_3$ at the inflexion point.
By Lemmas \ref{local-global} and \ref{lbs1}(2), the considered wall-crossing is described by the bifurcation of type $(2A_2\Longrightarrow A_5)$, in which a quartic with three real cusps turns into  quartic with one real and two complex conjugate cusps. Thus, the number $\#V_4^{irr,\RR}(3A_2,\bp(t))$ does not change.

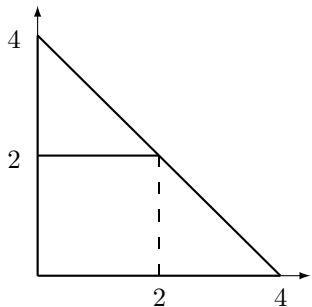
\begin{figure}
\setlength{\unitlength}{0.8cm}
\begin{picture}(6,5)(-2,0.5)
\thinlines
\put(1,1){\vector(0,1){4.5}}\put(1,1){\vector(1,0){4.5}}

\dashline{0.2}(3,1)(3,3)

\thicklines
\put(1,1){\line(1,0){4}}\put(1,1){\line(0,1){4}}
\put(1,5){\line(1,-1){4}}\put(1,3){\line(1,0){2}}

\put(2.9,0.5){$2$}\put(4.9,0.5){$4$}
\put(0.5,2.8){$2$}\put(0.5,4.8){$4$}

\end{picture}
\caption{Toric degeneration in the proof of Theorem \ref{tbs1new}(1)}\label{fbs2}
\end{figure}

At last, the only family of real non-reduced quartics of real dimension $7$ is formed by the unions of a conic and of a double line transversally intersecting the conic. We shall show that a generic curve of this kind cannot be deformed into a three-cuspidal quartic. Indeed, without loss of generality, we can consider affine quartics and can assume that the double line is just the $x$-axis, whereas the conical component is in general position. A possible one-parameter deformation of that non-reduced quartic is inscribed into a flat one-parameter family of surfaces, whose general fiber is $\PP^2$, and the central fiber is the union of the plane and of the toric surface associated with the trapezoid $T=\conv\{(0,0),(4,0),(0,2),(2,2)\}$ (see Figure \ref{fbs2}). However, there does not exist an affine curve with Newton polygon $T$ that has three cusps and transversally intersects the toric divisor $\Tor([(0,2),(2,2)])$ at two distinct points.

\section{Example: enumeration of three-cuspidal quartics revisited}
\label{secbs6new}

It is natural to modify the enumerative problem considered in the present paper in order to obtain real enumerative invariants.
Here we exhibit an enumerative invariant that counts real three-cuspidal quartics with the cusps in a fixed position. The problem reduces to the count of real conics tangent to all three coordinate axes. We plan to develop this observation and obtain a series of real enumerative invariants counting rational (or even non-rational) curves satisfying point constraints as well as certain relative constraints along an anti-canonical divisor.
Some related problems can be found in \cite{IS} and \cite[Section 5]{Sh23}.

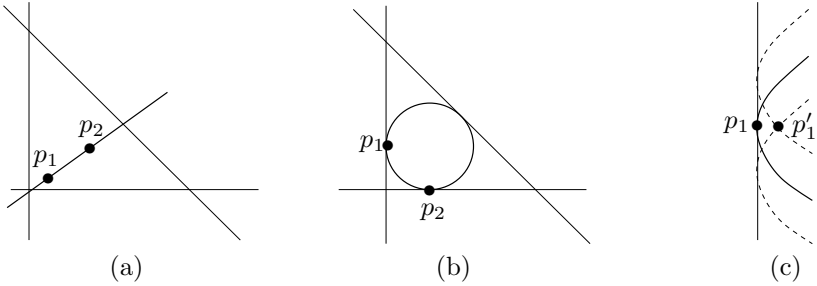
\begin{figure}
\setlength{\unitlength}{1cm}
\begin{picture}(12,4)(0,-0.4)
\includegraphics[width=0.9\textwidth, angle=0]{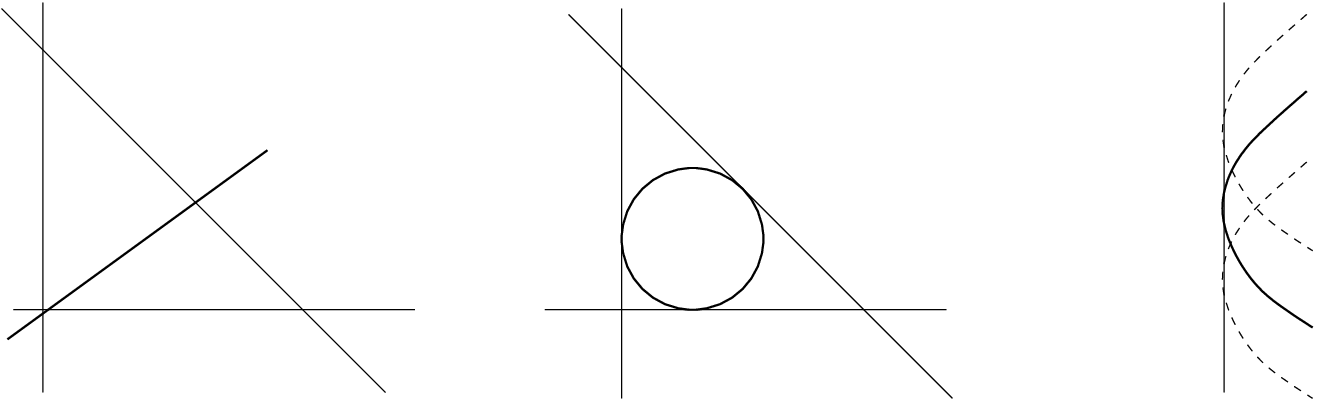}
\put(-10.2,0.8){$\bullet$}\put(-9.65,1.2){$\bullet$}\put(-5.72,1.24){$\bullet$}\put(-5.17,0.64){$\bullet$}
\put(-0.85,1.5){$\bullet$}\put(-0.57,1.49){$\bullet$}
\put(-10.3,1.1){$p_1$}\put(-9.7,1.5){$p_2$}\put(-6,1.3){$p_1$}\put(-5.2,0.4){$p_2$}
\put(-1.2,1.55){$p_1$}\put(-0.3,1.5){$p'_1$}
\put(-9.3,-0.4){(a)}\put(-5,-0.4){(b)}\put(-0.6,-0.4){(c)}
\end{picture}
\caption{Conics tangent to the coordinate axes} \label{fbs3}
\end{figure}

The family $V_4^{irr}(3A_2^{fix})$ of plane quartics having cusps at fixed non-collinear real point $z_1,z_2,z_3\in\PP^2$, is $T$-smooth of dimension $2$: indeed, that standard Cremona quadratic transformation with fundamental points $z_1,z_2,z_3$ takes the considered family into the family of conics that are tangent to each of the coordinate axes. Pick a triangle $T_+\simeq(\RR_+)^2$ in the complement of the three lines $\PP_\RR^2\setminus\left((z_1z_2)\cup(z_2z_3)\cup(z_3z_1)\right)$. Then we pick two points $p_1,p_2\in T_+$ and consider the problem of enumeration of real quartics in $V_4^{irr,\RR}(3A_2^{fix},(p_1,p_2))$. The above Cremona transformation converts this problem into the problem of enumeration of real conics passing through two points in the positive quadrant $(\RR_+)^2$ and tangent to each of the three coordinate axes. Recall that the complex answer is $4$ (see, for instance, \cite[Page 40]{Katz}). Let us verify that these (complex) conic all are real independently of the generic choice of $p_1,p_2\in T_+$. Consider the moduli space of stable unmarked conics $\overline{\mathcal M}_0(\PP^2,2)$, the subspace
$${\mathcal V}=\left\{[\bn:\PP^1\hookrightarrow\PP^2]\in{\mathcal M}_0(\PP^2,2)\ :\ \bn^*(\{x_i=0\})=2q_i,\ i=0,1,2\right\}$$ and its closure $\overline{\mathcal V}\subset \overline{\mathcal M}_0(\PP^2,2)$. It has dimension $2$, and the only one-dimensional strata in $\overline{\mathcal V}\setminus{\mathcal V}$ are formed by elements
$[\bn:\PP^1\to\PP^2]$ such that
\begin{itemize}\item $\bn_*\PP^2=2L$, where $L$ is a line passing through the intersection point of two coordinate axes, but different from these axes,
\item $\bn:\PP^1\to L$ is a double covering ramified at the intersection points with the coordinate axes.
\end{itemize}
The only possible realization of such a degeneration is exposed in Figure \ref{fbs3}(a).
Moving, say, $p_1$ away from $L$ yields two positions that are symmetric with respect to $L$, and hence the same number of real conics in count. To compute this number, we place $p_1,p_2$ on coordinate axes on the boundary of the positive quadrant: this degenerate problem, clearly, has a unique solution (see Figure \ref{fbs3}(b)). Shifting $p_1$ and $p_2$ inside the positive quadrant, we obtain two real local deformations in a neighborhood of each of $p_1,p_2$ Figure \ref{fbs3}(c), that are independent of each other, and hence the final answer is $4$.

\bmhead{Acknowledgments}

The author was partly supported by
the Bauer-Neuman Chair in Real and Complex Geometry. The author would like to thank Omer Bojan with whom we discussed the content of the paper and who helped much in analysis of deformations of singularities; however, finally he declined to be a coauthor. Special thanks are due to the unknown referee for the careful reading of the paper and for the criticism which helped me to make corrections and improve the readability.

\section*{Declarations}

\bmhead{Conflict of interests} The author declares that he has no conflict of interests.

\end{document}